\documentclass[12pt]{article}
\usepackage[francais,english]{babel}
\usepackage{amsmath}
\usepackage{amsfonts}
\usepackage{amssymb}
\usepackage{amsthm}
\usepackage{mhequ}
\usepackage{amscd}
\usepackage{fullpage}
\usepackage[matrix,arrow,ps]{xy}

\usepackage{hyperref}
\usepackage[sort=def]{glossaries}   

\usepackage[pdftex]{graphicx,color}

\newcommand{\Z}{\mathbb{Z}}
\newcommand{\R}{\mathbb{R}}

\newcommand{\U}{\mathbb{U}}
\newcommand{\E}{\mathbb{E}}

\newcommand{\mc}{\mathcal}
\newcommand{\mb}{\mathbb}

\newcommand{\eps}{\varepsilon}
\newcommand{\ind}{{\bf 1}}
\renewcommand{\P}{\mathbb{P}}

\newcommand{\vertiii}[1]{{\left\vert\kern-0.25ex\left\vert\kern-0.25ex\left\vert #1 
    \right\vert\kern-0.25ex\right\vert\kern-0.25ex\right\vert}}

\DeclareMathOperator{\Cov}{Cov}

\DeclareMathOperator{\Ber}{Ber}

\newtheorem{thm}{Theorem}[section]

\newtheorem{Prop}[thm]{Proposition}
\newtheorem{Lem}[thm]{Lemma}

\def\XXint#1#2#3{{\setbox0=\hbox{$#1{#2#3}{\int}$ }
\vcenter{\hbox{$#2#3$ }}\kern-.6\wd0}}

\title{Random clusters in the Villain and XY models}

\author{Julien Dub\'edat  \thanks{Department of Mathematics, Columbia University, 2990 Broadway, New York, NY 10027, USA.}  \and Hugo Falconet \thanks{Courant Institute, New York University, 251 Mercer street, New York, NY 10012, USA.}}

\date{\today}

\begin{document}

\maketitle

\begin{abstract}
In the Ising and Potts model, random cluster representations provide a geometric interpretation to spin correlations. We discuss similar constructions for the Villain and XY models, where spins take values in the circle, as well as extensions to models with different single site spin spaces. In the Villain case, we highlight natural interpretation in terms of the cable system extension of the model. We also list questions and open problems on the cluster representations obtained in this fashion. 
\end{abstract}


\section{Introduction}

The planar rotator model, or XY model, is a continuous spin model where single site spins take values in the unit circle $\mb{U}$ which exhibits the Berezinskii-Kosterlitz-Thouless phase transition \cite{Berezinsky:1970fr, Kosterlitz_1973}, a fact proved first by Fr\"ohlich and Spencer \cite{FS81} and revisited more recently in \cite{KP17, EM21, AHPS}. This result states that above a critical temperature $T_c$, the decay of spin correlation is exponential, whereas there is only a power law decay below $T_c$, i.e. $\langle \sigma_x \sigma_y \rangle_T = |x-y|^{- \alpha_T + o(1)}$
for some exponent $\alpha_T>0$ whenever $T < T_c$. Moreover, in this low temperature phase it is conjectured since the work of Fr\"ohlich and Spencer \cite{FS83} that the the XY model should behave at large scale like $e^{i T_{\mathrm{eff}} \Phi}$, where $\Phi$ is the planar Gaussian free field (GFF).
 
\medskip

The Villain model is a variant with the same phenomenology which is more tractable due to exact duality identities involving the integer-valued GFF. In this article, the Villain interaction is singled for a related but distinct reason, namely its relation with the so-called cable systems: there is a natural way to extend the spins defined on the vertices of the lattice to a continuous family of spins on the edges with a locally Brownian structure. Similar constructions were used in the context of isomorphism theorems between the Gaussian free field and the occupation field of trajectories by Lupu in \cite{Lupu_inter}, and then in the context of the first passage sets of the $2d$ GFF in \cite{ALS20}, the set of points in the domain that can be connected to the boundary by a path along which the GFF is greater than or equal to a fixed height.

\medskip

Our goal is to provide a geometric interpretation to spin correlations: first, for the Villain model in Proposition \ref{prop:Villain-k=1} and Proposition \ref{prop:Villain-k=2} below, by introducing random clusters in the extended Villain model, and then, for appropriate random cluster constructions in the XY model, which is the content of Theorem \ref{thm:general} below. In fact, this theorem covers more generally the O($2$) model, the XY and Villain models being special cases. Our constructions have the same spirit as the one of the Fortuin-Kasteleyn (FK)/random cluster representations of the Ising and Potts models (the joint distribution of random clusters with spins goes under the name of Edwards-Sokal coupling \cite{Edwards-Sokal} \cite[Section 1.4]{Grimmett}).  In the Ising model, two equal nearest neighbor spins are in the same cluster with a fixed probability, independently of other edges and the spins themselves can be retrieved from the geometrical clusters by attributing random signs to each cluster  \cite[Section 1.4]{Grimmett}. Then, \cite[Theorem 1.6]{Grimmett} states
$$
\langle \sigma_x \sigma_y \rangle_{T_c} \asymp \mb{P}( x \leftrightarrow y),
$$
where in the right-hand side $x$ is connected to $y$ if they are in the same cluster. In fact, they are equal up to a multiplicative constant, this doesn't specifically rely on being at the critical temperature and the same holds for the Potts model. Our motivation is twofold: first, akin to the FK model, these random clusters can be seen as a tool to study the continuous spin models; second, the FK model is a model interesting in itself (it provides a continuous interpolation of the $q$-Potts model in which $q$ is restricted to be an integer) and we believe the same is true in our case.

\medskip

The planar Ising model, with spins taking values only in $\{+1, -1\}$ has a phenomenology different than the one of the aforementioned continuous spin models. Rather, it is only at the critical temperature that the spin correlations have polynomial decay and below that temperature, the two-point correlation function may not vanish (as the separation goes to infinity), depending on the boundary conditions/Gibbs measure considered. Indeed, in this case and at low temperature, there is no uniqueness of translation invariant infinite-volume measure. Although we expect qualitative similarities, the interfaces of our random clusters in the KT phase and the interfaces in the critical Ising/FK model may also be different, and this is discussed in further details below.

\medskip

The Swendsen-Wang (SW) algorithm is commonly used to implement Monte Carlo simulations of the Ising spin model; at each step an FK cluster of spins (sampled conditionally on the spins) is flipped. More precisely, a step of the algorithm consists of the following: first, a bond is formed between every pair of nearest neighbors that are equal, with an explicit probability $p$ depending on the temperature and coupling constant of the model. Then, all spins in each cluster are flipped. Finally, all bonds are erased and the step is complete. The algorithm generalizes to the $q$-state Potts models in which each lattice site corresponds to a variable that can take $q$ different values.  One of these $q$ values is assigned with probability $1/q$ to each cluster and all variables in the cluster take this new value.  In \cite{Wolff}, when considering continuous spin models, Wolff replaced the global spin-inversion operation by a reflection operation in which only the connected component of one spin chosen uniformly at random is reflected over a randomly oriented plane. At each iteration, a new spin and an orientation of the plane are chosen.  A review of cluster algorithms in a general framework (which includes Potts model, random surfaces, continuous spins, among others) can be found in Section 2 of \cite{Omri-Peled}.

In the case of the XY model, Chayes considered a graphical representation based on the Wolff algorithm in \cite{Chayes-1}. He did so by writing spins  in the form $(\sigma_x \cos \theta_x , \tau_x \sin \theta_x)$ where $\sigma_x, \tau_x \in \{-1,1 \}$ and then by considering the FK-Ising model associated with one of these $\pm 1$ spins.  In particular, he proved that these percolation measures satisfy the FKG property and obtained a characterization of the positive magnetization of the system in term of percolation in the graphical representation (see also \cite{Chayes-2} for an extension to the Heisenberg model). The clusters he considered are the same as those in Lemma \ref{lem:preserve-xy-1}. In this work, in order to obtain result for higher spin observables, we also consider clusters made of correlated bonds (instead of only one bond as in \cite{Chayes-1, Chayes-2, Omri-Peled}).

\medskip

To conclude this introduction, we mention that the planar Villain and XY models have recently been the center of a regain of interest in the mathematics community, in particular with the the works \cite{Newman-Wu, Wirth19, Garban-reconstruction, Garban-quantitative} and the ones mentioned above. Furthermore,  some progress was made on the Fr\"ohlich Spencer conjecture as \cite{Bauerschmidt-dGFF-1} proved the convergence of the integer-valued Gaussian free field  towards the Gaussian free field, when the temperature is sufficiently large (which corresponds to low temperature continuous spin model). Ultimately, we refer the reader who wish to learn more about these spin models to \cite{FV-book, Peled-LN}, for an introduction to the field.

\medskip

This paper is organized as follows. First, we define the Villain model and cable systems in Section 2. Then, in Section 3 we prove some two sided estimates for some spin observables in terms of connectivity properties of random clusters. In Section 4, we generalize these results to  O($2$) models by introducing random clusters made of open/close bonds with appropriate joint distributions. In Section 5.1, we explain in which sense the usual random cluster model and its dilute version (related with the Blume-Capel-Potts model) can be seen from random clusters associated with similar cable systems but with a different continuous-time Markov chain. Finally, we list questions and open problems in Section 5.2.

\paragraph{Acknowledgments.} We wish to thank Ron Peled for pointing out several references.

\section{Villain model and cable systems}

\paragraph{Villain model.} The Villain model is a statistical mechanics model with spins taking values in the unit circle $\U\simeq\R/2\pi\Z$ and interactions given by a $\theta$ function. The specific choice of interaction was motivated by the following considerations. If one considers a low temperature $T$ (or $\beta$ large), then the XY weight $e^{\beta \cos (\theta_x-\theta_y)}$ can be approximated by $e^{- \frac{1}{2} \beta (\theta_x -\theta_y)^2}$ as the spins tend to be aligned in this regime, by a Taylor expansion. In order to preserve a model well-defined for angles modulo $2\pi$, it is natural to consider a periodized version of this interaction, leading to the Villain model. The reason one is interested in this approximation comes from the fact that the Fourier modes of this new interaction are Gaussian terms, which can be used to show powerful identities relating Villain observables to dual, integer-valued height models.

 More precisely, let ${\mc G}=(V,E)$ be a finite subgraph of $\Z^d$ (typically, a box or a torus). The probability measure on configurations  $(u_x)_{x\in V}=(e^{i\theta_x})_{x\in V}\in \U^V$ is given by
$\frac{1}{\mc Z}e^{-H(\theta)}\prod_{x\in V}d\theta_x$ where the energy is defined as
$$
e^{-H(\theta)}=\prod_{(xy)\in E}\sum_{n\in\Z}\exp\left(-\beta (\theta_x-\theta_y+2n\pi)^2\right).
$$
Here, $\beta>0$ plays the role of inverse temperature. 

We now turn to the Fourier decomposition of this interaction. To this end, recall the Jacobi $\Theta$ function ($\Im(\tau)>0$)
$$\Theta(z,\tau)=\sum_{n\in\Z}e^{i\pi n^2\tau}e^{2i\pi nz},$$
so that (for $t=4\pi\beta>0$, $v=(\theta_x-\theta_y)/2\pi$)
$$\Theta(itv,it)=e^{\pi tx^2}\sum_{n\in\Z}\exp(-\pi t (v+n)^2).$$
We have the functional equation (e.g. (8.9) in \cite{Chandra}) $\sqrt t\Theta(z, it)=e^{-\frac{\pi z^2}t}\Theta(\frac{z}{it},it^{-1})$ or
\begin{align}
\label{eq:low-high}
\sqrt t\sum_{n\in\Z}\exp(-\pi t (x+n)^2) & =\sum_{n\in\Z}e^{-\pi n^2 t^{-1}}\exp(2i\pi n x)  \nonumber \\
& =1+2\sum_{n=1}^\infty e^{-\pi n^2 t^{-1}}\cos(2\pi nx),
\end{align}
a standard instance of the Poisson summation formula. Note that the LHS (resp RHS) series converges faster when $t>1$ (resp. $t<1$).

\paragraph{Cable systems.}

While the Villain model was introduced for its special duality properties, we are exploiting here an indirectly related property. Observe that
\begin{equation}
\label{eq:heat-kernel}
p_t(\theta_1,\theta_2)=\frac{1}{\sqrt{2\pi t}}\sum_{n\in\Z}\exp\left(-\frac{1}{2t} (\theta_1-\theta_2+2n\pi)^2\right)
\end{equation}
where $(p_t)$ denotes the transition function for Brownian motion on $\U\simeq\R/2\pi\Z$. In particular, by Chapman-Kolmogorov:
$$
p_{t+s}(\theta_1,\theta_3)=\int p_t(\theta_1,\theta_2)p_s(\theta_2,\theta_3)d\theta_2.
$$
Then, the energy can be written as a product of transition functions, i.e.
\begin{equation}
\label{eq:energies}
e^{-H(\theta)}=(2\pi t)^{|V|/2}\prod_{(xy)\in E}p_t(\theta_x,\theta_y),
\end{equation}
with $1/2t=\beta$ ($t$ is thus a temperature parameter).

\smallskip

Now, we consider the cable system (also known as a metric graph) $\tilde{ \mc{G}}$ obtained by adding a segment of length $t(e)=t$ between $x$ and $y$, where $e=(xy)$ is any edge of ${\mc G}$ (i.e., any abstract edge in ${\mc G}$ becomes a line segment in $\tilde {\mc G}$). One can extend the Villain model from a measure on ${\mb U}^V$ to a measure on $C_0(\tilde {\mc G},\U)$ by sampling a $\mb{U}$-valued Brownian bridge (of duration $t$, from $\theta_x$ to $\theta_y$) to fill the values on these added intervals. This is modeled on the construction of the Gaussian Free Field (GFF) on cable systems by Lupu \cite{Lupu_inter} (where the single site spin space is $\R$, rather than $\U$ here). Earlier references on diffusions in cable systems include \cite{Baxter-Chacon,  Enriquez-Kifer, Folz}.

\smallskip

More generally, the construction works if the spin space is, say, a compact manifold  $S$ with a Riemannian metric whose associated volume form is $\mu$ and $(p_t)$ is the heat kernel for a symmetric diffusion on that manifold w.r.t. $\mu$. One can also consider $S$ non-compact if $\mu$ is finite; if the underlying graph ${\mc G}$ is finite, there is no restriction. The interaction is still given by $\prod_{(xy)\in E}p_t(u_x,u_y)$, the symmetry assumption implies $p_t(u_x,u_y) = p_t(u_y,u_x)$ and $(p_t)$ satisfies
\begin{equation}
\label{eq:CK}
p_{t+s}(u_x, u_y)=\int_S p_t(u_x, u_z)p_s(u_z, u_y)\mu(du_z).
\end{equation}
This property is used extend the distribution on vertices to the refined graphs with Kolmogorov extension (e.g, the midpoints of edges can be added by using $p_{t/2}$ for the interaction on each half-edge), and the consistency follows from  \eqref{eq:CK}. In the case of the GFF, $p_t$ is the heat kernel on $\mb{R}$ considered by Lupu and in the case of the Villain model described above, it is the heat kernel on $\mb{U}$.

The resulting extended  model $(X_v)_{v \in \tilde{\mc{G}}}$ provides a continuous stochastic process with sample paths in $C_0(\tilde {\mc G},S)$ which  satisfies a simple and strong Markov property as described in the lemma below.

In the following lemma, we consider two setups. The first one includes extended models with sample paths in $C_0(\tilde {\mc G},S)$, coming from symmetric diffusions. The second one includes Markov chains with finite state spaces, in which the extension comes from sampling Markov chain bridges. 

\begin{Lem}
Consider an extended model $(X_v)_{v \in \tilde{\mc{G}}}$ associated with a Markov process or chain as described above and a random connected compact set $K \subseteq \tilde{\mc{G}}$ such that $\{ K \subset U \}$ is measurable w.r.t. $(X)$ for all open sets $U$. Then,  conditionally on $(K, X_{| K})$, the distribution of $X_{|K^c}$ is described as an extended model with the same transition probabilities (i.e., any finite marginal is described by a nearest neighbor spin system whose edge energies are of the form \eqref{eq:energies} but depend on the length of the edges) with boundary condition $(X_v)_{v \in \partial K}$. 
\end{Lem}
In the above lemma, the sigma algebra associated with $X_{| \partial K}$ is that of $\cap_{\eps > 0} \sigma ( X_{K^{\eps} \backslash K} )$ where $K^{\eps}$ is an $\eps$-neighborhood of $K$. In particular, if $\partial K$ corresponds to the location of a jump in the case of the extension with Markov chains, although the value at the jump is ambiguous, the left-limit and the right-limit around it are not. This case does not occur for extensions using a diffusion. An example below will be the case of the dilute Ising model, say with some boundary conditions, for which the state space is $\{-1,0,1\}$ and where $K_x$ is the closure of $\{ y \in \tilde{\mc{G}} ~ | ~ \text{there is a path } \pi \text{ from } x \text{ to } y \text{ such  that } X_v = X_x \forall v \in \pi \}$. The set $\partial K_x$ consists exactly of jumps, can be seen as with two values (the limits from the inside and outside)  and of the boundary of $\tilde{\mc{G}}$. Any tiny neighborhood around the jumps will reveal the value of the boundary conditions (which will always be $0$ in the case of the dilute Ising model below). 


\smallskip

The strong Markov property was considered in the context of the GFF on metric graphs in Section 3 of \cite{Lupu_inter}. Also, still in this context, an example of a random compact set is the first passage set of level $-a$ considered in \cite{ALS20}, which is the set of points on the metric graph that are joined to the boundary by some path, on which the extended version of the GFF does not go below the level $-a$. The two-valued local sets considered in \cite{ASW,AS} are also very close to the set-up here.

\section{Random clusters in the Villain model}

In the Villain or XY models, there is a natural family of primary fields: $e^{ik\theta_x}$ for $k\in\Z$; this is especially of interest in the low temperature (KT) phase where it is expected that
\begin{equation}\label{eq:KTconj}
\langle \prod_j e^{ i k_{x_j} \theta_{x_j} } \rangle \propto \langle \prod_j  e^{i k_j \Phi_{x_j}} \rangle = \prod_{i<j} |x_i - x_j|^{- \alpha_T k_i k_j}
\end{equation}
where $\Phi$ is a planar Gaussian free field with some effective coupling constant fixing $\alpha_T$.

\medskip

Consider the following connectivity events given $\theta\in C(\tilde {\mc G},\U)$ a configuration in the extended Villain model: if $S$ is a subset of $\U$, and $x,y$ are points in $\tilde {\mc G}$, we write
\begin{equation}\label{eq:clusterdef}
x\stackrel{S}{\longleftrightarrow} y
\end{equation}
is there is a path $\gamma$ from $x$ to $y$ in $\tilde{\mc G}$ such that $\theta(\gamma)\subset S$. In other words, the clusters are the connected components of $\theta^{-1}(S)$. We will be in particular interested in the case $S=\{\pm e^{i\theta_0}\}^c$. 

\begin{Lem}
If $(xy)\in E$ and $t=t(xy)$, the probability that $u(z)\in \{\pm i\}^c$ for all $z\in [xy]$ given $u(x)=e^{i\theta_x}, u(y)=e^{i\theta_y}$ is:
\begin{align*}
\left\{\begin{array}{ll}
0&{\rm if\ } \cos(\theta_x)\cos(\theta_y)\leq 0\\
p_t^{[-\pi/2,\pi/2]}(\theta_x,\theta_y)/p_t(\theta_x,\theta_y)&{\rm if\ } \theta_x,\theta_y\in (-\pi/2,\pi/2)\\
p_t^{[-\pi/2,\pi/2]}(\theta_x-\pi,\theta_y-\pi)/p_t(\theta_x,\theta_y)&{\rm if\ } \theta_x,\theta_y\in (\pi/2,3\pi/2) \mod 2\pi
\end{array}
\right.
\end{align*}
\end{Lem}
The first case corresponds to two spins in the closure of the two connected components of $\mb{U} \backslash \{i, -i\}$.  Here $p_t^{[a,b]}$ denotes the heat kernel for linear Brownian motion in the segment $[a,b]$, absorbed on the boundary. By a repeated use of the reflection principle, one has (see. e.g. \cite{BMhandbook}, Appendix 1.6)
\begin{equation}\label{eq:linBMref}
p_t^{[-\pi/2,\pi/2]}(\theta_1,\theta_2)=\frac{1}{\sqrt{2\pi t}}\sum_{n\in\Z}\left(e^{-\frac{1}{2t}(\theta_1-\theta_2+2n\pi)^2}-
e^{-\frac{1}{2t}(\theta_1+\theta_2+(2n-1)\pi)^2}
\right).
\end{equation}
Hence the conditional probability can be written as
\begin{equation}
\label{eq:bond-proba-Villain}
\frac{f_t(\theta_1-\theta_2)-f_t(\theta_1+\theta_2-\pi)}{f_t(\theta_1-\theta_2)}  1_{\cos(\theta_x) \cos(\theta_y)>0}
=\frac{p_t(\theta_1,\theta_2)-p_t(\theta_1,\pi-\theta_2)}{p_t(\theta_1,\theta_2)}  1_{\cos(\theta_x) \cos(\theta_y)>0}
,
\end{equation}
where $f_t(x)=\sum_{n\in\Z}\exp(-\frac{1}{2t}(x+2n\pi)^2)$. Note that the second expression has a simple interpretation (and direct proof) in terms of a single application of the reflection principle (across the vertical axis) applied directly to the BM on $\U$ (rather then repeated applications of the reflection principle to its lift to $\R$).

If $x,y\in{\mc G}$, the conditional probability of $x\stackrel{\{\pm i\}^c}{\longleftrightarrow} y$ given $\theta_{|{\mc G}}$ is the probability that $x$ and $y$ are connected by a random cluster sampled as follows: for each $(xy)\in E$, the edge is open in the random cluster with probability \eqref{eq:c-constraint} independently of other edges.

\medskip

We consider an underlying graph ${\mc G}$ which is connected and with a non-empty subset $\partial\subset V$ designated as the boundary. We consider the Villain model with boundary conditions $\theta=1$ on $\partial$; $\langle\cdot\rangle$ denotes the expectation under the corresponding Villain measure. Under this measure, we remark that the distribution of $(u_x)_x$ is equal to that of $(\bar u_x)_x$, and consequently $\langle e^{i\theta_x}\rangle=\langle\cos(\theta_x)\rangle\in\R$ for all $x \in \tilde{\mc{G}}$.

The following proposition relates the asymptotics of the field $e^{i k \theta_x}$ with $k=1$ to those of the connectivity properties of the clusters of spins in $\{ \pm i \}^c$. 

\begin{Prop}
\label{prop:Villain-k=1}
For each $t_0 > 0$, there is $\delta=\delta(t_0,\deg(x))>0$ such that, if $t(xy)\geq t_0$ for all $y\sim x$, then
$$\delta\leq\frac{\langle \cos(\theta_x)\rangle}{\P(x\stackrel{\{\pm i\}^c}{\longleftrightarrow} \partial)}\leq 1.$$
\end{Prop}

The proof of the upper bound relies on exact cancellations using the strong Markov property of the extended Villain model and an appropriate reflection, and the lower bound follows from a local resampling of the spins nearby the site $x$, on the random cluster event.

\begin{proof}
Let $C\subset\tilde{\mc G}$ denote the cluster
$$C=\{y\in\tilde{\mc G}: y\stackrel{\{\pm i\}^c}{\longleftrightarrow} \partial\}$$
a random open subset of $\tilde{\mc G}$. Consider the transformation $\sigma:\Omega=C_0(\tilde{\mc G},\U)\rightarrow\Omega$ given by:
\begin{equation*}
\left\{\begin{array}{lll}
\sigma(u)_y&=u_y&{\rm if\ }y\in C\\
\sigma(u)_y&=-\overline{u_y}&{\rm if\ }y\notin C
\end{array}\right.
\end{equation*}
By the strong Markov property, $\sigma$ is a measure-preserving transformation (this can be seen as a reflection principle). Consequently,
$$\langle\cos(\theta_x)\ind_{x\notin C}\rangle=\langle\cos(\pi-\theta_x)\ind_{x\notin C}\rangle=0$$
which gives already the claimed upper bound. See the left part of Figure \ref{fig:villain-reflection} for an illustration.

\begin{figure}
\centering
\includegraphics[scale=1]{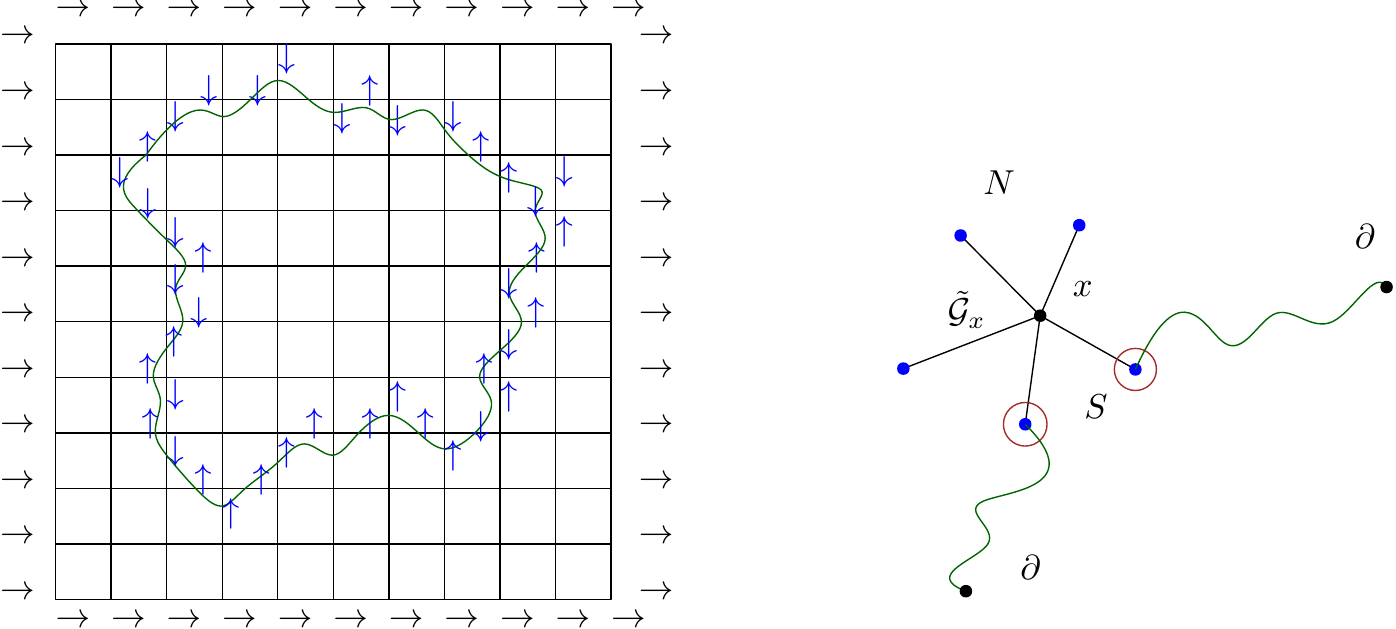}
\caption{Left figure: The inner boundary of the random cluster $C$ is in green and is obtained as an interpolation of the blue spins $\uparrow \downarrow$. Each spin inside this green loop is reflected over  $(Oy)$ by the map $\sigma$. Right figure: a 1-neighborhood of a vertex $x$ in ${\mc G}$; $y \in N$ are the blue disks, the star metric graph $\tilde{\mc{G}_x}$ has five branches, and $y\in S$ are further surrounded by red circles.}
\label{fig:villain-reflection}
\end{figure}

On the event $x\in C$, $\theta_x\in (-\frac\pi 2,\frac\pi 2)$ (because $\theta=0$ on $\partial$), and $\cos(\theta_x)>0$; we have to show that $\E(\cos(\theta_x)|x\in C)\geq \delta>0$.  We will prove this by using the local randomness near the site $x$.

Let $\tilde {\mc G}_x\subset\tilde {\mc G}$ be the union of segments connected to $x$, a star metric graph. Let $S$ be the set of neighbors of $x$ (in ${\mc G}$) that are connected to the boundary by a path outside of $\tilde{\mc G}_x$. Then $x\in C$ if and only if $S$ is non-empty, and $x$ is connected to $S$ in $\tilde {\mc G}_x$. See the right part of Figure \ref{fig:villain-reflection}.

Let $N$ be the set of neighbors of $x$.  Using that $\cos(\theta_x) \geq 0$ when $\{ x \in C \}$ occurs, we can bound from below
$$
\E( \cos(\theta_x) 1_{x \in C})  \geq \frac{1}{\sqrt{2}} \sum_{N \supset  s \neq \emptyset} \mb{P} (\theta_x \in (- \frac{\pi}{4}, \frac{\pi}{4}), x \leftrightarrow s, S=s ).
$$
Furthermore, since the conditional laws of the spins in $\tilde{\mc{G}}_x$ and the ones in $\tilde{\mc{G}} \backslash \tilde{\mc{G}}_x$ are independent given $(\theta_y : y \in N)$, we have
\begin{align*}
\mb{P} (\theta_x \in (- \frac{\pi}{4}, \frac{\pi}{4}), x \leftrightarrow s, S=s | (\theta_y)_{y \in N} ) & = \mb{P}_{\tilde{\mc G}_x} (\theta_x \in (- \frac{\pi}{4}, \frac{\pi}{4}), x \leftrightarrow s | (\theta_y)_{y \in N} ) \mb{P} ( S=s | (\theta_y)_{y \in N} )  
\end{align*}
Then we use 
$$
\mb{P}_{\tilde{\mc G}_x} (\theta_x \in (- \frac{\pi}{4}, \frac{\pi}{4}), x \leftrightarrow s | (\theta_y)_{y \in N} ) = \mb{P}_{\tilde{\mc G}_x} (\theta_x \in (- \frac{\pi}{4}, \frac{\pi}{4})| (\theta_y)_{y \in N},  x \leftrightarrow s  ) \mb{P}_{\tilde{\mc G}_x} (x \leftrightarrow s | (\theta_y)_{y \in N} )
$$
Hence, since
$$
\mb{P}(x\in C, S=s) = \E  \left( \P_{\tilde{\mc G}_x}(x\leftrightarrow s|(\theta_y)_{y\in N})\P(S=s|(\theta_y)_{y\in N})  \right),
$$
we just need a local estimate in $\tilde{\mc G}_x$ of the form
$$
\P_{\tilde{\mc G}_x}(\theta_x\in (-\frac \pi 4,\frac\pi 4)|x\leftrightarrow s,(\theta_y)_{y\in N})\geq\delta,
$$
uniformly in $(\theta_y)_{y\in N}$ (with $\theta_y\in (-\frac\pi 2,\frac\pi 2)$ if $y\in s$).

This boils down to standard estimates of the form:
$$\P_x(B_t\in (c,d)| \forall s\in [0,t]{\rm , }B_s\in (a,b))\geq \delta>0$$
uniformly in $t\geq t_0>0$ and $x\in (a,b)$, where $a<c<d<b$ and $B$ is standard Brownian motion (started from $x$ under $\P_x$). This can be e.g. read off explicit formulae of the form \eqref{eq:linBMref}.
\end{proof}

\paragraph{Remark.} With free boundary conditions or on the torus, for $x,y\in{\mc G}$, one can fix $\theta_y=0$ using global rotational invariance and set $\partial=\{y\}$ to obtain
$$
\langle e^{i(\theta_x-\theta_y)}\rangle\asymp \P(x\stackrel{\{\pm ie^{i\theta_y}\}^c}{\longleftrightarrow} y).
$$ 

\medskip

Proposition \ref{prop:Villain-k=1} dealt with the case $k=1$ for which $e^{i k \theta_x } = \sigma_x$. In the following proposition, we provide a geometric interpretation to $e^{i 2 \theta_x} = \sigma_x^2$ in which clusters of spins in the half-circles $(\xi, -\xi)$ and $(\bar{\xi}, -\bar{\xi})$ arise where $\xi:=e^{i\pi/4}$. The set-up  is as before. 
\begin{Prop}
\label{prop:Villain-k=2}
For each $t_0 > 0$, there is  $\delta=\delta(t_0,\deg(x))>0$ such that, if $t(xy)\geq t_0$ for all $y\sim x$, then
$$
\delta\leq\frac{\langle \cos(2\theta_x)\rangle}{\P(x\stackrel{\{\pm \xi\}^c}{\longleftrightarrow} \partial, x\stackrel{\{\pm \bar\xi\}^c}{\longleftrightarrow} \partial)}\leq 1.
$$
\end{Prop}

The scheme of the proof is similar to that of Proposition \ref{prop:Villain-k=1}. However, finding the appropriate random cluster connection event in order to obtain exact cancellations via reflections becomes more involved.

\begin{figure}
\centering
\includegraphics[scale=1]{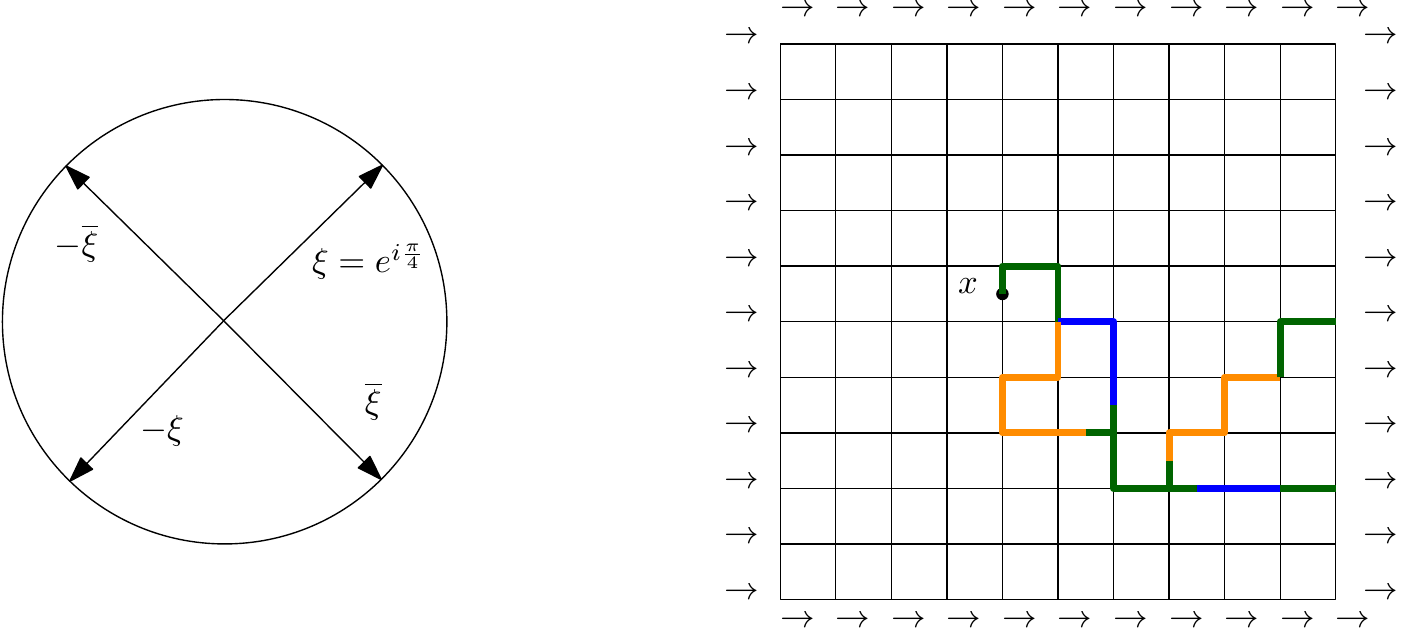}
\caption{Left figure: the maps $z \mapsto \xi^2 \overline{z}$ and  $z \mapsto \overline{\xi}^2 \overline{z}$  are  the reflection across $(\pm \xi)$ and  $(\pm \overline{\xi})$ involved in the definition of $\sigma_1$ and $\sigma_2$.   Right figure: The point $x$ belongs to the intersection of the clusters $C_1$ and $C_2$. The colors blue, green, and orange denote spins with values in the south, east and north quadrant respectively. The path in green contains spins in $[-\pi/4, \pi/4]$ and then splits in points in $C_1$ (in blue) and $C_2$ (in orange).}
\label{fig:k=2}
\end{figure}

\begin{proof}
Consider the clusters
\begin{align*}
C_1&=\{y\in\tilde{\mc G}: y\stackrel{\{\pm \xi\}^c}{\longleftrightarrow} \partial\}\\
C_2&=\{y\in\tilde{\mc G}: y\stackrel{\{\pm \bar \xi\}^c}{\longleftrightarrow} \partial\}
\end{align*}
and the measure-preserving transformations $\sigma_1,\sigma_2$:
\begin{align*}
\sigma_1(u)_x=u_x{\rm\ if\ }u\in C_1{\rm,\ \ }&\sigma_1(u)_x=\xi^2\overline{u_x}{\rm\ otherwise}\\
\sigma_2(u)_x=u_x{\rm\ if\ }u\in C_2{\rm,\ \ }&\sigma_2(u)_x={\bar \xi}^2\overline{u_x}{\rm\ otherwise}
\end{align*}
Remark that the set $\{x:u_x\in\{\pm\xi\}\}$ is preserved both by $\sigma_1$ and $\sigma_2$ (and similarly for $\{x:u_x\in\{\pm\bar\xi\}\}$); in particular both $\sigma_1$ and $\sigma_2$ preserve the events $x\in C_1$, $x\in C_2$ and their complements. Moreover, if $x\notin C_i$, $\Re((\sigma_i(u)_x)^2)=-\Re(u_x^2)$. It follows that
$$\langle\cos(2\theta_x)\ind_A\rangle=0$$
if $A$ is one of three exclusive connection events: $x\notin C_1, x\in C_2$; or $x\in C_1,x\notin C_2$; or $x\notin C_1,x\notin C_2$. So
$$
\langle\cos(2\theta_x)\rangle=\langle\cos(2\theta_x)\ind_{x\in C_1\cap C_2}\rangle.
$$
Moreover, on the event $\{x\in C_1\cap C_2\}$, $\theta_x\in (-\frac\pi 4,\frac \pi 4)$ and $\cos(2\theta_x)>0$. For the lower bound one argues as in the proof of Proposition \ref{prop:Villain-k=1}.
\end{proof}

\paragraph{Cluster-swapping algorithm.}

The type of conditional spin-flip measure-preserving transformations also provides a cluster-swapping algorithm for sampling the Villain model, akin to the classical Swendsen-Wang algorithm \cite{SW-algorithm} for the Ising model. This is closely related to the Wolff algorithm \cite{Wolff}

The underlying graph is ${\mc G}$ (typically a box in $\{1,\dots, N\}^d\subset\Z^d$, or a box $(\Z/N\Z)^d$). The boundary conditions are: $\theta=0$ on the boundary $\partial$ (e.g. for a box), or free (e.g. for a torus).

One step of the algorithm consists of:
\begin{enumerate}
\item Pick an angle $\nu$ uniformly in $[0,2\pi]$.
\item For each edge $(xy)\in E$, declare the edge open with probability:
\begin{equation*}
\left\{\begin{array}{ll}
g_t(\theta_x-\nu,\theta_y-\nu)&{\rm\ if }\cos(\theta_x-\nu)>0{\rm\ and\ }\cos(\theta_y-\nu)>0\\
g_t(\pi+(\theta_x-\nu),\pi+(\theta_y-\nu))&{\rm\ if }\cos(\theta_x-\nu)<0{\rm\ and\ }\cos(\theta_y-\nu)<0\\
0&{\rm\ otherwise}
\end{array}\right.
\end{equation*}
\item For each cluster $C$ (connected component of open edges):
\begin{itemize}
\item (case with boundary) if $C$ is connected to the boundary, do nothing
\item otherwise, with probability 1/2, do nothing; with probability 1/2,  for all $x\in C$ replace $\theta_x$ with $2\nu+\pi-\theta_x\mod 2\pi$.
\end{itemize}
\end{enumerate}
Repeat until (close enough to) equilibrium. 

Remark that the step is equivalent to the following. Sample $\tilde\sigma$, the extension of $\sigma=e^{i\theta}$ to the cable system, conditionally on $\sigma$; resample $\tilde\sigma$ conditionally on $\Im(e^{-i\nu}\tilde\sigma)$; restrict it back to ${\mc G}$. In other words, the role of extended configuration played by the Edwards-Sokal coupling for the Swendsen-Wang algorithm is played here by the cable system extension.

For $\theta_1,\theta_2\in (-\frac\pi 2,\frac\pi 2)$, 
$$g_t(\theta_1,\theta_2)=\frac{f_t(\theta_1-\theta_2)-f_t(\theta_1+\theta_2-\pi)}{f_t(\theta_1-\theta_2)}$$
where $f_t(x)=\sum_{n\in\Z}\exp(-\frac{1}{2t}(x+2n\pi)^2)$, an even, $2\pi$-periodic function. If $t$ (a fixed temperature parameter) is small ($t<2\pi$) (resp. large, $t\geq 2\pi$), the LHS of \eqref{eq:low-high} (resp. RHS of \eqref{eq:low-high}) converges faster. In the first case, with $x\in [-\pi,\pi]$, a truncation of the series to $n \in \{-2, \dots, 2\}$ gives a good approximation; in the second case, a truncation to $n \in \{1, 2\}$.

\section{Generalization to the O($2$) model}

\paragraph{Definition of the model and examples.} We consider a spin model on a graph $\mc{G}$ where spins $(u_x)_{x \in \mc{G}}$ take values on the unit circle with interactions given by
\begin{equation}
\label{eq:def-model}
\mb{P}(du)  \propto \prod_{(xy) \in E} w(u_x, u_y)\prod_xd\theta_x
\end{equation}
Here $w(u,v) \in (0,\infty)$ and $u_x=e^{i\theta_x}$. We assume that $w$ is invariant under rotations and reflections. Typically, we want to favor configurations of spins that are aligned so it is natural to assume that $w(\cdot, \cdot)$ is continuous and that $\theta \mapsto w(1,e^{i \theta})$ is increasing on $[-\pi,0]$ and decreasing on $[0,\pi]$.

\medskip

Let us mention some examples: when $w$ is the heat kernel on the unit circle \eqref{eq:heat-kernel}, this corresponds to the Villain model. Furthermore, taking $w(u_x, u_y) = e^{\beta u_x \cdot u_y }$ gives the classical XY model. More generally, this covers the O($2$) model (see, e.g. \cite[Section 9]{FV-book}).

\medskip

In what follows, we denote by $R$ the reflection w.r.t the imaginary axis (for $z \in \mb{C}$, $R(z) = - \bar{z}$). We will be interested in clusters of spins that lie in the same half circles $\mb{U}_L := e^{i [-\pi/2,\pi/2]}$, $\mb{U}_R := e^{i [\pi/2,3\pi/2]}$, built by open/closed bonds. We will add extra randomness to this model by conditionally sampling open/closed bonds independently of each others, given the spins on $\mc{G}$. We  fix the boundary conditions to be $u_x = 1$ for $x \in \partial \mc{G}$. 

\medskip

In a context of numerical simulations to estimate physical observables of the critical Ising and XY models, Wolff \cite{Wolff} introduced an algorithm that revolves around cluster reflections that generalize the Swendsen-Wang algorithm for Potts spins \cite{SW-algorithm}. The algorithm considers $O(n)$ spins models (for which spins take values in $S^{n-1}$ and the interaction is given by $e^{\beta \sigma_x \cdot \sigma_y}$). On a lattice with periodic boundary conditions, a uniform reflection $R$ and a uniform site $x$ are chosen. The spin at $x$ is reflected via $R$ and so is its cluster, made of bonds open with explicit probabilities given the value of the spins (see \cite[Equation (5)]{Wolff}). The construction of the clusters below (from the next paragraph up to Lemma \ref{lem:preserve-xy-1}) provide an extension of Wolff's considerations.

\paragraph{Cluster distribution and swapping map $\sigma_x$.}

We motivate here the choice of extended model $(u,e)$ where $u$ refers to the above spin model and $e$ refers to a collection of bonds $e = (e_{xy})_{(xy) \in E}$ we will consider. We say that $(xy)$ is open if $e_{xy} = 1$ and close if $e_{xy} = 0$. 

 We look for a probability distribution on bonds $p(u_x,u_y) = \mb{P}( e_{xy} =1)$ which satisfies $p(u_x,u_y) = 0$ if $u_x$, $u_y$ lie in the two different sides of $\mb{U}$, $p(R(u), R(v)) = p(u,v)$ for every $u,v$ and such that the following maps $\bar{\sigma}_z : (u,e) \mapsto (\sigma_z(u,e),e)$ are measure-preserving for any point $z \in \mc{G}$: if $z$ is connected to the boundary, then $\sigma_z(u,e) = u$. If this is not the case, then $\sigma_z(u,e)_y = R(u_y)$ for each $y \in C_z$, the cluster containing $z$, and $\sigma_z(u,e)_y = u_y$ for every $y \in \mc{G} \setminus C_z$. This transformation doesn't affect the clusters but only the spins.
 
Then, we will consider the random cluster model  sampled as follows: conditionally on $(u_x)_{x\in \mc{G}}$, for each $(xy)\in E$, the edge is open in the random cluster with probability $p(u_x,u_y)$ independently of other edges.
 
\medskip

We denote by $P$ the distribution on spins $u$ and bonds $e$ obtained by the combination of $w$ and $p$. Clearly, if $z$ is connected to the boundary in $u$, then $dP(\bar{\sigma}_z(u,e))/dP(u,e) = 1$. If this is not the case, then
\begin{equation}
\label{eq:invariance}
\frac{dP(\bar{\sigma}_z(u,e))}{dP(u,e)} = \prod_{x \in C_z, y \in \mc{G} \setminus C_z, x \sim y} \frac{w(R(u_x),u_y)}{w(u_x,u_y)} \frac{1-p(R(u_x),u_y)}{1-p(u_x,u_y)}.
\end{equation}
If $u_x, u_y$ are in opposite sides, $p(u_x,u_y) = 0$ and
$$
\frac{w(R(u_x),u_y)}{w(u_x,u_y)} \frac{1-p(R(u_x),u_y)}{1-p(u_x,u_y)}  = 1 \Rightarrow p(R(u_x),u_y) = 1 - \frac{w(u_x,u_y)}{w(R(u_x),u_y)},
$$
which is in $(0,1)$ from the assumptions on $w$ (positivity, finiteness and monotonicity). The other case is symmetric: if $u_x, u_y$ are in the same side, $p(R(u_x),u_y) = 0$ and
$$
\frac{w(R(u_x),u_y)}{w(u_x,u_y)} \frac{1-p(R(u_x),u_y)}{1-p(u_x,u_y)}  = 1 \Rightarrow p(u_x,u_y) = 1 - \frac{w(R(u_x),u_y)}{w(u_x,u_y)}.
$$
These two conditions are the same given that $R^2 = \mathrm{Id}$. 

This choice of probability for opening/closing bonds, i.e. 
$$
p(u_x,u_y) = ( 1 - \frac{w(R(u_x),u_y)}{w(u_x,u_y)} ) 1_{u_x, u_y \text{ are in the same side of } \mb{U}},
$$
gives $\frac{dP(\bar{\sigma}_z(u,e))}{dP(u,e)} =1$ so that, for any function $F : \mb{U}^{\mc{G}} \times \{0,1\}^{E} \to \mb{R}$,
$$
\E( F(\bar{\sigma}_z(u,e))) =  \int F(\bar{\sigma}_z(u,e)) dP(u,e) = \int F(\bar{\sigma}_z(u,e)) dP(\sigma_z(u,e))  = \E( F(u,e)).
$$
We record this in the following lemma. Also, we note that \eqref{eq:proba-bond} below is a generalization of \eqref{eq:bond-proba-Villain} to general weights.

\begin{Lem}
\label{lem:preserve-xy-1}
Consider the extended model with random clusters $(u,e)$ where $u=(u_x)_{x \in \mc{G}}$ is as in \eqref{eq:def-model} and $e = (e_{xy})_{(xy) \in E} \subset \{0,1\}^E$ is obtained by opening edges (i.e., setting $e_{xy} =1$) independently of other edges given the spins with probability  
\begin{equation}
\label{eq:proba-bond}
p(u_x,u_y) = ( 1 - \frac{w(R(u_x),u_y)}{w(u_x,u_y)} ) 1_{u_x, u_y \text{ are in the same side of } \mb{U}}.
\end{equation}
Then, for every $z \in \mc{G}$, $(u,e) \mapsto \bar{\sigma}_z(u,e)$ is measure preserving.
\end{Lem}

We observe that the clusters considered in this lemma, conditionally on the spins, coincide with those of Chayes in \cite{Chayes-1}. Indeed, with his notation $(\cos \theta_x, \sin \theta_x) = (a_x \sigma_x, b_x \tau_x)$ with $\sigma_x, \tau_x \in \{-1,1\}$. When $\cos(\theta_x) \cos(\theta_y) > 0$,
$$
\frac{w(R u_x, u_y)}{w( u_x, u_y)} = \frac{ e^{\beta \cos(\pi-\theta_x-\theta_y)}}{e^{\beta \cos(\theta_x-\theta_y)}} = e^{- 2 \beta \cos(\theta_x) \cos(\theta_y)} = e^{- 2 \beta a_x a_y}.
$$
In his paper, this term corresponds  to the FK expansion of the Ising model associated to the $\sigma_x$'s (see equation (7) in \cite{Chayes-1}). These clusters are also included in the general exposition of Cohen-Alloro and Peled, see the paragraph on spins $O(n)$ models in Section 2 of \cite{Omri-Peled}.

\smallskip

With the same assumptions as above, we have
\begin{Prop}
\label{prop:XY-k=1}
There is $\delta=\delta(\deg(x))>0$ such that,
$$
\delta\leq\frac{\langle \cos(\theta_x)\rangle}{\mb{P}(x \stackrel{C}{\longleftrightarrow} \partial)}\leq 1.
$$
\end{Prop}

This is a generalization of Proposition \ref{prop:Villain-k=1}. However, the strong Markov property  of the extended Villain model on edges is lost and we instead rely on the swapping map of the previous lemma.

\begin{proof}
We start with the upper bound.  We use the measure preserving map $\bar{\sigma}_x$ (Lemma  \ref{lem:preserve-xy-1}) to obtain 
$$
\E( \Re(u_x) 1_{C_x \cap \partial \mc{G} = \emptyset}) = \E( - \Re(\overline{u_x}) 1_{C_x \cap \partial \mc{G} = \emptyset})  = 0.
$$
Here, $C_x \cap \partial \mc{G} = \emptyset$ is equivalent to $x \overset{C}{\nleftrightarrow} \partial$ so the spin $u_x$ is reflected. In combination with $\theta_x \overset{(d)}{=} \overline{\theta_x}$, the upper bound follows.

For the lower bound, the same scheme as in the proof of Proposition \ref{prop:Villain-k=1} applies (instead of using the local randomness in the metric graph $\tilde{\mc{G}}_x$ we use  the randomness of the adjacent bonds). This time, the result boils down to a lower bound on 
$$
\mb{P} (\theta_x \in [-\pi/4, \pi/4] ~ | ~ x \leftrightarrow s,  (\theta_y)_{y \in N}).
$$

Here, instead of relying on Brownian motion estimates, this is based on the following. Suppose the graph has only one edge $(xy)$ and $u_y$ is fixed in $\mb{U}_R$.  We want a lower bound on
$$
\mb{P} ( u_x \in e^{i [-\pi/4, \pi/4]} ~ |~  (xy) \text{ is open }, u_y ),
$$
which is uniform in $u_y \in e^{i [-\pi/2,\pi/2]}$. This follows from the expression (Bayes formula)
$$
\mb{P} ( d u_x  ~ |~  (xy) \text{ is open }, u_y ) = \frac{w(u_x, u_y) p(u_x,u_y)}{\int_{\mb{U}_R} w(v,u_y) p(v,u_y) \lambda(dv)} du_x,
$$
where $\lambda$ is the Lebesgue measure on the unit circle $\mb{U}$. This implies, by integrating the above equation and given the explicit expression of $p(u_x,u_y)$ in \eqref{eq:proba-bond},
$$
\mb{P} ( u_x \in e^{i [-\pi/4, \pi/4]} ~ |~  (xy) \text{ is open }, u_y )  = \frac{ \mb{P}(u_x \in e^{i [-\pi/4, \pi/4]} ~ | ~ u_y) - \mb{P}(u_x \in R e^{i [-\pi/4, \pi/4]} ~ | ~ u_y)}{\mb{P}(u_x \in \mb{U}_R ~ | ~ u_y) - \mb{P}(u_x \in \mb{U}_L ~ | ~ u_y)}.
$$
This function is continuous and positive in $u_y \in \mb{U}_R$. Indeed, it is clear away from $u_y = \pm i$. Furthermore, as $u_y \to i$,
$$
\frac{\int_{e^{i [-\pi/4, \pi/4]}} w(u, u_y) - w(u, R u_y) \lambda(du)} {\int_{e^{i [-\pi/2, \pi/2]}} w(u, u_y) - w(u, R u_y) \lambda(du)} \to \frac{\int_{e^{i [-\pi/4, \pi/4]}} \frac{d}{dh}_{|h=0} w(u, i e^{ih}) \lambda(du)}{ \int_{e^{i [-\pi/2, \pi/2]}} \frac{d}{dh}_{|h=0} w(u, i e^{ih}) \lambda(du) } >0
$$
as $w(u, i e^{-ih}) - w(u, i e^{ih}) = -2h \frac{d}{dh}_{|h=0} w(u, i e^{ih}) + o(h)$, from the strict monotonicity of $w$.

Note that when $w(u_x,u_y) = \rho(\cos(\theta_x-\theta_y)) $ with $\rho$ increasing, the RHS is equal to $\frac{\rho(\sqrt{2}/2)-\rho(-\sqrt{2}/2)}{\rho(1)-\rho(-1)}$ when $u_y = i$.

The general case is treated similarly. We roughly sketch it and omit the details. In this case, with $s \neq \emptyset$, we write
$$
\mb{P} ( du_x ~ |~  x \leftrightarrow s, (u_y)_{y \in N} ) = \frac{\prod_{y \sim x} w(u_x,u_y) \prod_{z \in s} p(u_x, u_z)}{\int_{\mb{U}_R} \prod_{y \sim x} w(u,u_y) \prod_{z \in s} p(u, u_z) \lambda(du)} du_x.
$$
There we integrate in $e^{i[-\pi/4,\pi/4]}$. The $y$'s in $N \backslash s$ do not cause concern. For those in $s$,  we use $w(u_x, u_y) p(u_x,u_y) = w(u_x,u_y)-w(u_x,Ru_y)$ and proceed as above.
\end{proof}

\paragraph{Clusters correlation.} The clusters we introduced in the Villain model are all measurable with respect to the extended model with the cable systems. There are two clusters in Proposition \ref{prop:Villain-k=2} that naturally live in the same space (in term of randomness). To consider two random clusters associated with pairs of open/close bonds for the model \ref{eq:def-model}, we need to prescribe their joint distribution. The choice we make is motivated by the explicit joint distribution for the Villain model, written in a more general fashion.

\medskip

We still use the expression of the bond probability from \eqref{eq:proba-bond}, i.e.
$$
p(u_x,u_y) = ( 1 - \frac{w(R(u_x),u_y)}{w(u_x,u_y)} ) 1_{u_x, u_y \text{ are in the same side}}
$$
when considering reflections associated with two half-circles. We will consider the clusters $C_1$ (associated with the half-circles separated by $\xi$, $-\xi$, with reflection $R=R_1$) made of what we call \textit{bonds 1}, denoted by $e_{xy}^1$, and $C_2$ (associated with the half-circles separated by $\overline{\xi}$, $-\overline{\xi}$, with reflection $R=R_2$), made of \textit{bonds 2}, denoted by $e_{xy}^2$. We also denote by $C_i^z$ the connected component of $C_i$ which contains $z$.

To define them on the same probability space, it is enough to impose the covariance or correlation between the two events. If $X \sim \Ber(p)$, $Y \sim \Ber(q)$, then  $c = \mathbb{P}(X=1, Y=1)$ determines the joint distribution of $X$ and $Y$ as well. (Note also that $\Cov(X,Y) = c - pq$.)
\begin{align*}
\mb{P}(X=1,Y=1) & = c \\
\mb{P}(X=1,Y=0) & = p - c \\
\mb{P}(X=0,Y=1) & = q - c  \\
\mb{P}(X=0,Y=0) & = 1 - p - q + c
\end{align*}
These equalities immediately imply the constraints:
\begin{equation}
\label{eq:c-constraint}
\max(0,p+q-1) \leq c \leq \min(p,q)
\end{equation}
and any $c$ is this range can be chosen. 

In the case of the Villain model we can compute the corresponding value of $c$ above explicitly. We have in this case
$$
w(u_x,u_y) = p_t^{\mb{U}}(\theta_1, \theta_2) , \qquad \text{where } u_x = e^{i \theta_1}, u_y = e^{i \theta_2}.
$$
Recalling $\xi = e^{i \pi/4}$, with $R_1(u_x) = \xi^2 \overline{u_x}$ (the reflection fixing $\xi$ and $-\xi$),  and $R_2(u_x) = \overline{\xi}^2 \overline{u_x}$ (the one fixing $\bar{\xi}$ and $-\bar{\xi}$), we observe that
\begin{align}
& p_{xy} = \mb{P}(e_{xy}^1 =1) = ( 1 - \frac{w(R_1 u_x,u_y)}{w(u_x,u_y)} ) 1_{u_x, u_y \text{ are in the same side of } \{ \pm \xi \}^c}, \label{eq:cluster-C1} \\
& q_{xy} = \mb{P}(e_{xy}^2=1) = ( 1 - \frac{w(R_2 u_x,u_y)}{w(u_x,u_y)} ) 1_{u_x, u_y \text{ are in the same side of } \{ \pm \overline{\xi} \}^c},   \label{eq:cluster-C2}
\end{align}
and
\begin{equation}
\label{eq:def-c}
c = \mb{P}(e_{xy}^1=1, e_{xy}^2=1)  = \frac{p_t^{[-\pi/4,\pi/4]}(\theta_1,\theta_2)}{p_t^{\mb{U}}(\theta_1,\theta_2)}  1_{u_x, u_y \in [-\pi/4, \pi/4]\mod \pi/2}.
\end{equation}

Furthermore, by using the reflection principle and splitting the following series with $k$ odd and $k$ even, we have
\begin{align*}
p_t^{[-\pi/4,\pi/4] }  = & \sum_{k \in \mb{Z}} p_t(\theta_1, \theta_2 + k \pi) - p_t(\theta_1, -\theta_2 + k \pi - \pi/2) \\
= & \sum_{k \in \mb{Z}} p_t(\theta_1, \theta_2 + 2k \pi) + \sum_{k \in \mb{Z} }  p_t(\theta_1, \theta_2 + 2k \pi+ \pi)  \\
& - \sum_{k \in \mb{Z}} p_t(\theta_1+\theta_2+2k\pi+\pi/2) -  \sum_{k \in \mb{Z}} p_t(\theta_1+\theta_2+2k\pi-\pi/2) \\
= & w(u_x,u_y) + w(u_x, - u_y) - w(R_1(u_x),u_y)  - w(R_2(u_x),u_y),
\end{align*}
since
\begin{align*}
w(R_1(u_x),u_y) & = w(e^{i ( \pi/2-\theta_1)}, e^{i \theta_2} ) = \sum_{k \in \mb{Z}} p_t(\theta_1+\theta_2+2k \pi - \pi/2), \\
w(R_2(u_x),u_y) & = w(e^{-i ( \pi/2+\theta_1)}, e^{i \theta_2} ) = \sum_{k \in \mb{Z}} p_t(\theta_1+\theta_2+2k \pi + \pi/2).
\end{align*}
Hence, in the Villain model, $c$ defined as in \eqref{eq:def-c} can be expressed as 
\begin{equation}
\label{eq:c-value}
c = (1  - \frac{w(R_1 u_x ,u_y)}{w(u_x,u_y)} - \frac{w(R_2 u_x ,u_y)}{w(u_x,u_y)} + \frac{w(u_x, - u_y)}{ w(u_x,u_y)}) 1_{u_x, u_y \in [-\pi/4, \pi/4]\mod \pi/2}.
\end{equation}
Alternatively, this can also be seen by using directly the reflection principle (or method of images) for Brownian motion on the circle. Furthermore, we note that $R_1 R_2 = R_2 R_1 = - \mathrm{Id}$.

\medskip

\textbf{Remark.}  For a general weight function as in \eqref{eq:def-model}, as soon as $c$ defined in \eqref{eq:def-c} satisfies $c \geq 0$, the other constraints in \eqref{eq:c-constraint} are all satisfied. Indeed, if $u_x, u_y$ are not in the same connected component of $\mb{U} \backslash \{\pm \xi, \pm \bar{\xi} \}$, $c_{xy} = 0 \leq p_{xy}$. When they are in the same component, $p_{xy} \geq c_{xy}$ from the fact that $w(R_2 u_x, u_y) \geq w(-u_x,u_y)$ (the spins in the first term are in adjacent connected components, whereas they are in opposite components in the second one). Similarly, we obtain $q_{xy} \geq c_{xy}$. Finally, if the spins are not in the same component, either $p_{xy} = 0$ or $q_{xy} = 0$ and $p_{xy}+q_{xy} - 1 \leq 0$. Else, it is equal to $c_{xy} - w(-u_x,u_y)/w(u_x,u_y) \leq c_{xy}$.

\bigskip

\textbf{Remark.} In the XY model, the non-negativity of $c$ in \eqref{eq:c-value} boils down to the following inequality: for every $\theta_x, \theta_y \in [-\pi/4,\pi/4]$,
$$
e^{\beta \cos(\theta_x-\theta_y)} + e^{-\beta \cos(\theta_x-\theta_y)} \geq e^{\beta \sin(\theta_x+\theta_y)} + e^{-\beta \sin(\theta_x+\theta_y)}.
$$
For any $w$, this becomes an equality when $\theta_x \in \{-\pi/4, \pi/4\}$, for every $\theta_y$, since one reflection fixes $u_x$ and the other maps it to $-u_x$. By a change of variables $u = \theta_x - \theta_y$, $v = \theta_x + \theta_y$, we see that 
$$
\frac{d}{d u} (e^{\beta \cos u} + e^{-\beta \cos u }) = - \beta \sin u \ (e^{\beta \cos u} - e^{- \beta \cos u} ), \quad \frac{d}{d v} (e^{\beta \sin v} + e^{-\beta \sin v }) =  \beta \cos v \ (e^{\beta \sin v} - e^{- \beta \sin v} )
$$
This is equal  to zero when $u = 0 \mod \pi$ or $u = \pi/2 \mod \pi$ and the same result holds for $v$. The case $u = 0$ corresponds to $\theta_x = \theta_y$, and $v = 0$ gives $\theta_x = - \theta_y$. In this case, we indeed have $e^{\beta} + e^{-\beta} \geq 2$.

In the case of a more general increasing function $\rho : [-1, 1] \to (0, \infty)$ and $w(u_x,u_y) =\rho(\cos(\theta_x-\theta_y))$, this becomes
$$
\frac{d}{d u} \rho(\cos u) = - \sin u \ ( \rho'(\cos u) - \rho'(-\cos u) ), \quad  \frac{d}{d v} \rho(\sin v) = \cos v \ ( \rho'(\sin v) - \rho'(- \sin v) ).
$$
If $\rho$ is strictly convex, $\rho'$ is increasing and the condition boils down again to $\rho(1) + \rho(-1) \geq 2 \rho(0)$, which is also satisfied by the convexity assumption.

\medskip

For the following Proposition \ref{prop:XY-k=2}, Lemma \ref{lem:XY-preserved-map} and Theorem \ref{thm:general}, we consider an extended model (i.e., with additional pairs of bonds) for which $c$ defined in \eqref{eq:c-value} satisfies the constraints \eqref{eq:c-constraint} with $p,q$ respectively given in \eqref{eq:cluster-C1} and \eqref{eq:cluster-C2}.  

\begin{Prop} There is $\delta=\delta(\deg(x))>0$ such that, 
\label{prop:XY-k=2}
$$
\delta\leq\frac{\langle \cos(2\theta_x)\rangle}{\P(x\stackrel{C_1}{\longleftrightarrow} \partial, x\stackrel{C_2}{\longleftrightarrow} \partial)}\leq 1.
$$
where $C_1$ alone (resp. $C_2$) corresponds to clusters of bonds that are open with probability $p$ as in \eqref{eq:cluster-C1} (resp. $q$ as in \eqref{eq:cluster-C2}), with correlation determined by the parameter $c$ in \eqref{eq:c-value}.
\end{Prop}

To prove this proposition, we need to introduce the map $\sigma_1^z$ which reflects the spins of the cluster $C_1^z$  using $R_1$ when $z$ is not connected to $\partial$ through $C_1$ (equivalently, when $C_1^z \cap \partial \mc{G} = \emptyset$). The map $\bar{\sigma}_2^z$ is defined analogously with $C_2$ and $R_2$. We also introduce $\bar{\sigma}_i^z : (u,e) \mapsto (\sigma_i^z(u,e),e)$. Here, $e = ((e_{xy}^1)_{(xy) \in E}, (e_{xy}^2)_{(xy) \in E})$.

We note that if two neighboring spins in $C_1$ are flipped using $R_1$ then: if they are in the same arc of $\mb{U}$ for $C_2$, they still are after reflection by $R_1$; if there are not in the same arc for $C_2$, this is also preserved after reflection by $R_1$.

\begin{Lem}
\label{lem:XY-preserved-map}
The maps $\bar{\sigma}_1^z$ and $\bar{\sigma}_2^z$ defined on the space $((u_x)_{x \in \mc{G}}, (e_{xy}^1)_{(xy) \in E}, (e_{xy}^2)_{(xy) \in E})$ are measure preserving.
\end{Lem}

We postpone the proof of Lemma \ref{lem:XY-preserved-map} and prove first Proposition \ref{prop:XY-k=2}.

\begin{proof}[Proof of Proposition \ref{prop:XY-k=2}] 

The upper bound relies on the exact cancellations coming from the invariance under $\bar{\sigma}_i^z$ (Lemma \ref{lem:XY-preserved-map}) with the reflections of the spins. Again, we emphasize that $e$ is unaffected by these maps and only spins may be affected. Therefore, connectivity properties of the random clusters remain unchanged, and spins are flipped. The remaining details are in the proof of Proposition \ref{prop:Villain-k=2}.

The lower bound is obtained as before. 
\end{proof}

\smallskip

Altogether, we have
\begin{thm}
\label{thm:general}
Consider an O($2$) model as in \eqref{eq:def-model}, and suppose furthermore that the weight function $w (e^{i \theta_x}, e^{i \theta_y})$ is given by the Villain model or by $\rho(\cos(\theta_x-\theta_y))$ with $\rho$ strictly convex. Then, there exists extended models $((u_x)_{x \in \mc{G}}, (e_{xy})_{(xy) \in E})$ and $((u_x)_{x \in \mc{G}}, (e_{xy}^1)_{(xy) \in E}, (e_{xy}^2)_{(xy) \in E})$ with random clusters $C$ and $(C_1, C_2)$ such that
$$
\delta\leq\frac{\langle \cos(\theta_x)\rangle}{\mb{P}(x \stackrel{C}{\longleftrightarrow} \partial)}\leq 1, \qquad
\delta\leq\frac{\langle \cos(2\theta_x)\rangle}{\P(x\stackrel{C_1}{\longleftrightarrow} \partial, x\stackrel{C_2}{\longleftrightarrow} \partial)}\leq 1,
$$
where $\delta$ depends only on the degree of $x$ and on the weight function $w$.

Furthermore, the assumptions include the case of the XY model and the distribution of the extensions with bonds coincides with the ones associated with the extended Villain model with cable systems.
\end{thm}

\paragraph{Remark.} Let us mention that the convexity assumption in the above theorem is used only to ensure the term $c$  in \eqref{eq:def-c} is non-negative. We could instead simply make that assumption.

\begin{proof}
The result follows by combining Propositions \ref{prop:Villain-k=1}, \ref{prop:Villain-k=2}, \ref{prop:XY-k=1}, and \ref{prop:XY-k=2}.
\end{proof}

Now, we come back to the proof that was skipped and which is the only remaining part to obtain Theorem \ref{thm:general}.

\begin{proof}[Proof of Lemma \ref{lem:XY-preserved-map}]
We have to show that the map $\bar{\sigma}_1^z$ associated with reflecting spins of the cluster $C_1^z$  using $R_1$  (when $z$ is not connected to $\partial$ through $C_1$) is measure-preserving. The result for the map $\bar{\sigma}_2^z$ follows by symmetry.

To this end, we will make use of the exact expression of $c$ in \eqref{eq:c-value} and of $p_{xy}, q_{xy}$. As in the proof of Lemma \ref{lem:preserve-xy-1}, we want the analogue of\eqref{eq:invariance} to be equal to one. However, in this equation only one case arose whereas now we have to take into account several possibilities.

To simplify the notation, we set
\begin{align*}
b_{xy}^0 & := \mb{P}(e_{xy}^1=0, e_{xy}^2=0) = 1 - p_{xy} - q_{xy} + c_{xy}, \\
b_{xy}^1 & := \mb{P}(e_{xy}^1=1, e_{xy}^2=0) = p_{xy} - c_{xy}, \\
 b_{xy}^2 & := \mb{P}(e_{xy}^1=0, e_{xy}^2=1) =  q_{xy} - c_{xy}.
\end{align*}

Spins and bonds outside of $C_1^z$ are not affected by $\sigma_1^z$ and the same holds for bonds between vertices in $C_1^z$: $a)$ if both bonds are open (with probability $c_{xy}$), after reflection $R_1$ of the spins, by direct inspection the probability \eqref{eq:c-value} is not affected. $b)$ In the case where only bond 1 is open (with probability $b_{xy}^1$) the probability remains unchanged by $R_1$ since $p_{xy}$ in \eqref{eq:cluster-C1} does not change.

\medskip

Most of the remaining work lies in the case of edges $(xy)$ at the boundary of $C_1^z$ where $x \in C_1^z$ and $x \sim y \notin C_1^z$. We distinguish all possible cases. It could be useful to  the reader to look at Figure \ref{fig:preserve} when following the arguments, although most cases are not covered in the figure.

\begin{figure}
\centering
\includegraphics[scale=1]{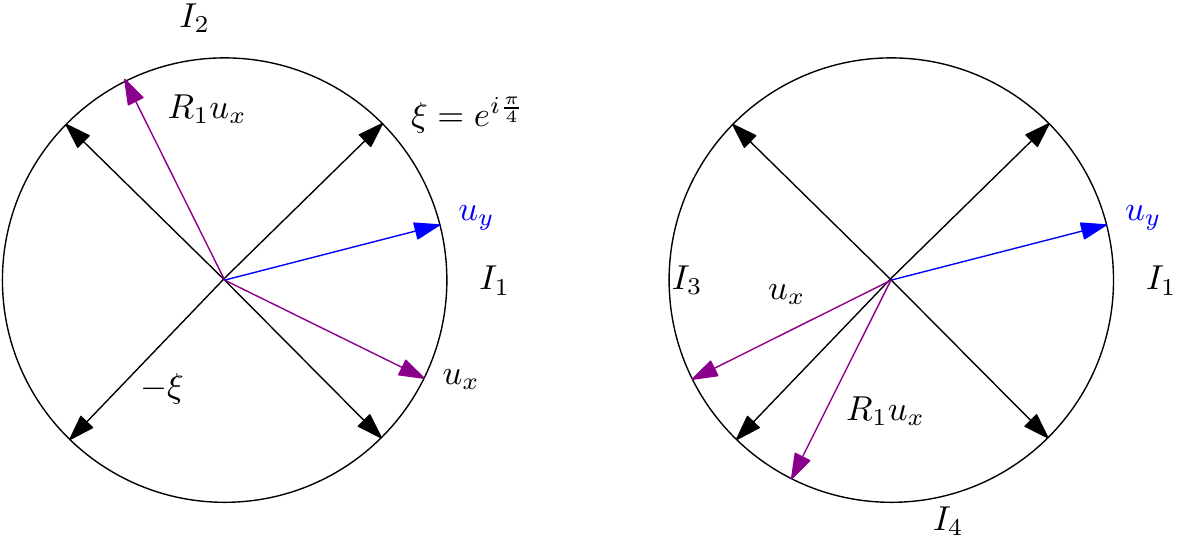}
\caption{Some of the cases and reflections involved in the proof of Lemma \ref{lem:XY-preserved-map}} 
\label{fig:preserve}
\end{figure}

\textit{Cases where both bonds are closed.} There are 4 cases (others follow by symmetries). We partition $\mb{U}$ in four arcs of equal size, $I_1, I_2, I_3, I_4$, ordered counterclockwise and with $I_1 = [-\pi/4,\pi/4] \mod 2\pi$ (see Figure \ref{fig:preserve}). We show below that, for every case , the following relation holds,
$$
\frac{w(R_1 u_x, u_y)}{w(u_x,u_y)} \frac{b^0_{R_1 u_x, u_y}}{b^0_{u_x, u_y}} = 1.
$$

Suppose $u_x, u_y \in I_1$. Only $u_x$ is reflected and $R_1 u_x \in I_2$ so $p_{xy} = c_{xy} = 0$ for the new configuration $\sigma_1^z(u)$ so this reduces to
$$
b^0_{R_1 u_x, u_y}  = 1 - ( 1 - \frac{w(R_2 R_1 u_x,u_y)}{w(R_1 u_x,u_y)} ) =  \frac{w(R_2 R_1 u_x,u_y)}{w(R_1 u_x,u_y)}.
$$
The result follows from
\begin{align*}
b^0_{u_x, u_y} = & 1 - ( 1 - \frac{w(R_1 u_x ,u_y)}{w(u_x,u_y)} ) - ( 1 - \frac{w(R_2 u_x ,u_y)}{w(u_x,u_y)} ) \\
& +  (1  - \frac{w(R_1 u_x ,u_y)}{w(u_x,u_y)} - \frac{w(R_2 u_x ,u_y)}{w(u_x,u_y)} + \frac{w(u_x, - u_y)}{ w(u_x,u_y)}) \\
= & \frac{w(R_2 R_1 u_x, u_y)}{ w(u_x,u_y)}).
\end{align*}

Suppose $u_x \in I_2, u_y \in I_1$. Then $R_1 u_x \in I_1$ so from the equalities of the previous case (in reverse order)
$$
b^0_{R_1 u_x, u_y} = \frac{w(R_2 R_1 R_1 u_x, u_y)}{ w(R_1 u_x,u_y)}) =  \frac{w(R_2 u_x, u_y)}{ w(R_1 u_x,u_y)}), \qquad
b^0_{u_x, u_y}  = \frac{w(R_2  u_x,u_y)}{w(u_x,u_y)}.
$$

Suppose $u_x \in I_3, u_y \in I_1$. Then $R_1 u_x \in I_4$ so $q_{xy} = c_{xy} = 0$ for the new configuration $\sigma_1^z(u)$ so this reduces to
$$
b^0_{R_1 u_x, u_y}  = 1- ( 1 - \frac{w(R_1 R_1u_x ,u_y)}{w(R_1 u_x,u_y)} ) = \frac{w(u_x ,u_y)}{w(R_1 u_x,u_y)}.
$$
For the configuration $u$, $p_{xy} = q_{xy} = c_{xy} = 0$ so $b^0_{u_x, u_y}  = 1$ and the result follows.

Finally, suppose $u_x \in I_4, u_y \in I_1$. The computations follow by the results of the previous case and we obtain
$$
b^0_{R_1 u_x, u_y} = 1, \qquad b^0_{u_x, u_y} = \frac{w(R_1 u_x,u_y)}{w(u_x ,u_y)}.
$$

\textit{Cases where only the bond 2 is open.}
Similarly, we want to show that 
$$
\frac{w(R_1 u_x, u_y)}{w(u_x,u_y)} \frac{b^2_{R_1 u_x, u_y}}{b^2_{u_x, u_y}} = 1.
$$
Here, two cases occur modulo symmetries. 

Suppose $u_x, u_y \in I_1$. Then
\begin{align*}
b_{u_x, u_y}^2 & = ( 1 - \frac{w(R_2 u_x ,u_y)}{w(u_x,u_y)} )  -  (1  - \frac{w(R_1 u_x,u_y)}{w(u_x,u_y)} - \frac{w(R_2 u_x ,u_y)}{w(u_x,u_y)} + \frac{w(u_x, - u_y)}{ w(u_x,u_y)})  \\
& = \frac{w(R_1 u_x ,u_y)}{w(u_x,u_y)} - \frac{w(R_1 R_2 u_x, u_y)}{ w(u_x,u_y)},
\end{align*}
and since $R_1 u_x \in I_2$, $c_{xy} = 0$ for $\sigma_1^z(u)$ so the result follows from
$$
b_{R_1 u_x, u_y}^2 = ( 1 - \frac{w(R_2 R_1 u_x ,u_y)}{w(R_1 u_x,u_y)} ).
$$

Now, suppose $u_x \in I_2, u_y \in I_1$. Using the intermediate results of the previous case, we find
$$
b_{u_x, u_y}^2 =  ( 1 - \frac{w(R_2 u_x ,u_y)}{w(u_x,u_y)} )
$$
and
$$
b_{R_1 u_x, u_y}^2 =  \frac{w(R_1 R_1 u_x ,u_y)}{w(R_1 u_x,u_y)} - \frac{w(R_1 R_2 R_1 u_x, u_y)}{ w(R_1 u_x,u_y)} = \frac{w(u_x ,u_y)}{w(R_1 u_x,u_y)} - \frac{w(R_2 u_x, u_y)}{ w(R_1 u_x,u_y)}.
$$
This concludes the proof.
\end{proof}

\section{Extensions, questions and open problems}

\subsection{Dilute Potts model}

Following \cite{Dilute-Potts}, the dilute Potts model is a model of spins taking values in $\{1, \dots, Q\}$ with vacancies (represented by the state $0$), whose distribution is given by
\begin{equation}
\label{eq:dilute-potts}
\P((\sigma_x)_x)\propto e^{K \sum_{ x \sim y} \delta_{\sigma_x, \sigma_y} (1- \delta_{\sigma_x \sigma_y, 0})  + V \sum_{ x \sim y}\delta_{\sigma_x \sigma_y, 0} + D \sum_x \delta_{\sigma_x,0}},
\end{equation}
where $\delta_{a,b} = 1$ if $a=b$ and zero otherwise. It generalizes the Potts model and admits a random-cluster representation \cite[Equation (5)]{Dilute-Potts}, as the Potts model itself  \cite{Grimmett}. It also generalizes the Blume-Capel model, a version of the Ising model with vacancies introduced in \cite{Blume, Capel}, and more generally the Blume-Capel-Potts model, which was studied in the mathematics literature in \cite{Grimmett-Blume-Capel} (see, e.g., \cite[Section 5]{Dilute-Potts}). We refer the reader to \cite{Grimmett-Blume-Capel, Dilute-Potts} and the references therein for background on these models, simulations, and their conjectured phase diagrams.

The goal of this section is to explain in which sense the usual random cluster model (i.e., for Potts) can be seen from random clusters in the extension of the Potts model to cable systems and that the models arising from this extension fall in \eqref{eq:dilute-potts}.

\smallskip

\paragraph{The Markov chain and the cable system.} We consider a continuous-time Markov chain with $Q+1$ states: $Q$ spins labeled by $\{1, \dots, Q\}$ and an empty state labeled by $0$. The process evolves as follows: when in the empty state, it  jumps to any other state at rate $\lambda$ and when in a spin state, it jumps to the empty state with rate $1$. The invariant probability measure of the system is given by
\begin{equation}
\label{eq:invariant-measure}
\mu(dy) = \frac{\lambda}{1+Q \lambda} \delta_{y \neq 0} + \frac{1}{1+Q \lambda} \delta_{y = 0},
\end{equation}
as detailed balance for the measure $\mu$ is satisfied, $\mu_1 q_{1 0} = \frac{\lambda}{1+Q \lambda}  \times 1  = \frac{1}{1+Q\lambda} \times \lambda = \mu_0 q_{0 1} $. 

Now, we consider $p_t(x,y)$ the transition density w.r.t. $\mu$, i.e.
$$
P_t f(x) = \E_x [ f(X_t)] = \sum_{y} \mb{P}_x(X_t = y) f(y)   = \int p_t(x,y) f(y) \mu(dy),
$$
which is a symmetric function, i.e., $p_t(\sigma_x, \sigma_y) = p_t(\sigma_y, \sigma_x)$ as
\begin{equation}
\label{eq:density-function}
p_t(\sigma_x, \sigma_y) = \mu(\sigma_y)^{-1} \mb{P}(X_t = \sigma_y | X_0 = \sigma_x) = \mu(\sigma_x)^{-1} \mb{P}(X_t = \sigma_x | X_0 = \sigma_y).
\end{equation}
From Chapman-Kolmogorov $P_{t+s} f(x) = P_t (P_s f)(x)$, we find $p_{t+s}(x,y) = \int p_t(x,z) p_s(z,y) \mu(dz)$ hence the extension to a cable system.

Below, we consider the measure on spins
\begin{equation}
\label{eq:measure-1}
P(\sigma) \propto \prod_{x \sim y} p_t(\sigma_x, \sigma_y) \prod_x \mu(d \sigma_x).
\end{equation}
More generally, one can also consider
\begin{equation}
\label{eq:measure-2}
P_u(\sigma) \propto u^{N_0(\sigma)} \prod_{x \sim y} p_t(\sigma_x, \sigma_y) \prod_x \mu(d \sigma_x).
\end{equation}
where $N_0(\sigma)$ is the number of vacancies in the spin configuration. The measures in \eqref{eq:measure-2}  give another parametrization of those in \eqref{eq:dilute-potts}. Indeed, the space of the $\log$ of the weights associated with two neighbors is $4$ dimensional: equal  spins, different spins, one vacancy and one spin, and two vacancies. Here, one can write $u^{N_0(\sigma)} = \prod_{x \sim y} u^{(\delta_{\sigma_x}/\mathrm{deg(x)}+\delta_{\sigma_y}/\mathrm{deg(y)})}$, where $\mathrm{deg}(x)$ is the degree of the vertex $x$. The fourth parameter is obtained by multiplying \eqref{eq:dilute-potts} or \eqref{eq:measure-2} by $e^{A}$ for a normalization parameter $A$.

\paragraph{Explicit transition densities.} We provide an explicit expression of $p_t(\sigma_x, \sigma_y)$. Consider the matrix whose entry $q_{i j}$ is  given by the rate to go from state $i$ to state $j$ and where $i=0$ corresponds to empty spin. The eigenvalues of this matrix are $0, -Q \lambda - 1$ both with multiplicity $1$, and $-1$ with multiplicity $Q-1$. So, each transition probability is of the form $\alpha e^{-(Q\lambda +1) t} + \beta e^{-t} + \gamma$ for every $t \geq 0$. To identify $(\alpha, \beta, \gamma)$, we use that $t \to \infty$ gives the invariant probability measure $\mu$, that $t = 0 $ gives either $0$ or $1$, and for $t$ close to $0$, the rates $q_{i,j}$ give the first order in $t$.  Therefore,
$$
\gamma = \mu(\sigma_y), \qquad \alpha + \beta + \gamma = \delta_{\sigma_x, \sigma_y}, \qquad - \alpha (Q \lambda +1) - \beta = q_{\sigma_x, \sigma_y},
$$
and we solve in $(\alpha, \beta)$,
$$
\begin{bmatrix}
1 & 1  \\
-(1+Q \lambda) & -1
\end{bmatrix}
\begin{bmatrix}
\alpha \\
\beta
\end{bmatrix}
= 
\begin{bmatrix}
\delta_{\sigma_x, \sigma_y} - \mu(\sigma_y) \\
q_{\sigma_x, \sigma_y}
\end{bmatrix},
\qquad
\begin{bmatrix}
\alpha \\
\beta
\end{bmatrix}
=
\frac{1}{Q \lambda} \begin{bmatrix}
- \delta_{\sigma_x, \sigma_y} + \mu(\sigma_y) - q_{\sigma_x, \sigma_y}\\
(1+ Q \lambda) (\delta_{\sigma_x, \sigma_y} - \mu(\sigma_y) )+ q_{\sigma_x, \sigma_y}
\end{bmatrix}
$$
We find, for $\sigma_x \sigma_y \neq 0$ and $\sigma_x \neq \sigma_y$,
\begin{align*}
\mb{P}_{\sigma_x}(X_t = \sigma_y) & = \frac{1}{Q(1+Q\lambda) } e^{- (Q \lambda+1)t} - \frac{1}{Q} e^{-t} + \frac{\lambda}{1 + Q \lambda} \\
 \mb{P}_{\sigma_x}(X_t = 0) & =  \frac{1}{1+ Q \lambda}(1- e^{-(Q \lambda+1) t} ) \\
\mb{P}_{\sigma_x}(X_t = \sigma_x) & = \frac{1}{Q(1+Q\lambda)} e^{-(Q \lambda+1)t} + \frac{Q-1}{Q} e^{-t} +  \frac{\lambda}{1 + Q \lambda}  
\end{align*}
and
$$
 \mb{P}_{0}(X_t = 0)  =  \frac{1}{1+ Q \lambda} (1+Q \lambda e^{- (Q\lambda+1)t}), \qquad \mb{P}_{0}(X_t = \sigma_x)  = \frac{\lambda}{1+ Q \lambda}   (1-e^{-(Q \lambda+1)t}).
$$
We note that for the transitions with an empty state, a reduced chain shows directly the absence of $e^{-t}$ term. Combining theses expressions with \eqref{eq:invariant-measure} and \eqref{eq:density-function} gives
\begin{align*}
p_t(\sigma_x,\sigma_x) & = 1 - \frac{1}{Q \lambda}  e^{-(1+Q\lambda)t} +  \frac{1+\lambda Q}{\lambda} \frac{Q-1}{Q} e^{-t},  & p_t(\sigma_x,0)  = 1 - e^{-(1+Q\lambda)t},  \\
p_t(\sigma_x,\sigma_y) & = 1 + \frac{1}{Q \lambda}  e^{-(1+Q\lambda)t} -  \frac{1+\lambda Q}{\lambda} \frac{1}{Q} e^{-t}, & p_t(0,0)  = 1 + Q \lambda e^{-(1+Q\lambda)t}.
\end{align*}
The probabilistic representation shows that each of them is positive, for every $\lambda, t >0$.

\paragraph{Clusters: correlations and connectivity.} We define the clusters to be the connected components of $\{ x : \sigma_x = 0\}^c$. There are $Q$ types of clusters, one for each spin. Associated to these is a natural percolation model where a bond is open if its extremities have the same spin and are not separated by a zero (i.e. the Markov chain bridge has not jumped), and closed otherwise (see Figure \ref{fig:extended}).  Given the values of the spins on the vertices, a bond between two neighbors $x$ and $y$ is closed if $\sigma_x \neq \sigma_y$ or $\sigma_x \sigma_y = 0$, it is otherwise closed with probability  
$$
\mb{P}_{\sigma_x} (T_0 < t | X_t = \sigma_x) = \frac{ \mb{P}_{\sigma_x}(T_0 < t)}{\mb{P}_{\sigma_x}(X_t = \sigma_x)} \mb{P}_{\sigma_x} (X_t = \sigma_x | T_0 < t) = \frac{1-e^{-t}}{\mb{P}_{\sigma_x}(X_t = \sigma_x)} \frac{1}{Q},
$$
where $T_0 = \inf \{t \geq 0 ~ : ~  X_t = 0 \}$. 

\begin{figure}
\centering
\includegraphics[scale=1]{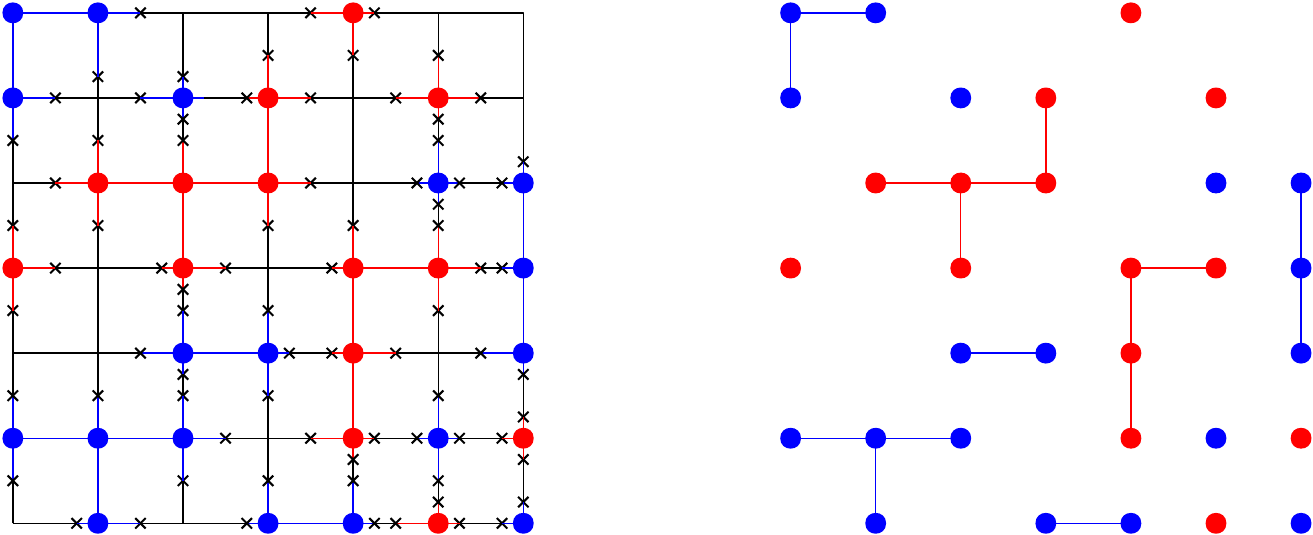}
\caption{Left: an extended configuration with two spins (blue and red) and the empty state (black); crosses mark jumps of the Markov chain bridges. Right: the associated random clusters}
\label{fig:extended}
\end{figure}

With this notion of clusters and percolation model, we have the following relation between correlations of spins and connectivity properties:
\begin{equation}
\tau_{x,y} := \E( \delta_{\sigma_x,\sigma_y} | \sigma_x \sigma_y \neq 0)  - \frac{1}{q} = (1-q^{-1}) \mb{P}(x \leftrightarrow y | \sigma_x \sigma_y \neq 0)
\end{equation}
Indeed,  $\E( \delta_{\sigma_x, \sigma_y} | \sigma_x \sigma_y \neq 0)  = \E( 1 \cdot 1_{x \leftrightarrow y} + \delta_{\sigma_x, \sigma_y} 1_{x \nleftrightarrow y} | \sigma_x \sigma_y \neq 0 )$ and the result follows from
\begin{align*}
\mb{P}(\sigma_x = \sigma_y, T_0 < t | \sigma_x \sigma_y \neq 0 ) & = \mb{P}(\sigma_x = \sigma_y | T_0< t, \sigma_x \sigma_y \neq 0) \mb{P}(T_0 < t | \sigma_x \sigma_y \neq 0) \\
& = \frac{1}{q} \mb{P} ( x \nleftrightarrow y | \sigma_x \sigma_y \neq 0).
\end{align*}
In the case of the random cluster model, this is \cite[Theorem (1.16)]{Grimmett}, and in the diluted random cluster one, this is \cite[Equation (3.9)]{Grimmett-Blume-Capel}.

\paragraph{The limit $\lambda \to \infty$ and the Potts model.}
When $\lambda \to \infty$, the invariant measure $\mu$ becomes the uniform distribution on $\{ 1, \dots, Q \}$ and for $\sigma_x \neq \sigma_y$,
$$
p_t(\sigma_x,\sigma_x)  = 1 + (Q-1) e^{-t}, \qquad p_t(\sigma_x, \sigma_y) = 1 - e^{-t}.
$$
In order to relate this limiting model with the usual random cluster model, we look for $(\alpha, \beta)$ such that
$$
p_t(\sigma_x, \sigma_y) = e^{\alpha + \beta \delta_{\sigma_x, \sigma_y}}, \qquad \beta =  \ln \frac{1+(Q-1)e^{-t}}{1-e^{-t}}, \qquad e^{-  \beta} =   \frac{1-e^{-t}}{1+(Q-1)e^{-t}}
$$
where $\beta$ is the inverse temperature of the model. To identify this extended model with the usual coupling between the $Q$-Potts model and the random cluster model, we need to verify that bonds are open only between neighboring spins with the same values with probability $1-e^{-\beta}$. When $\lambda \to \infty$ we get $\lim_{\lambda \to \infty} \mb{P}_{\sigma_x}(X_t = \sigma_x) = \frac{1}{Q} + (1-\frac{1}{Q}) e^{-t}$ hence
$$
p =  1 - \frac{1-e^{-t}}{1+(Q-1)e^{-t}} = \frac{Q e^{-t}}{1+(Q-1)e^{-t}} = 1 - e^{- \beta}
$$
and we indeed retrieve the FK model (see Theorem (1.13) part (b) in \cite{Grimmett}). A generalization of this description of the coupling in the case of the Blume-Capel-Potts and diluted random cluster models can be found in the paragraphs following the proof of Theorem 3.7 in \cite{Grimmett-Blume-Capel}.

\subsection{Heisenberg and $O(N)$ models for $N \geq 3$}

 Similar results hold for the extension of the Villain model corresponding to $O(N)$ models for which $N \geq 3$ (the Brownian motion on the unit circle for the Villain model now takes values in the $N-1$ sphere) or the Heisenberg model (whose edge weight is given by $w(u_x,u_y) = e^{\beta u_x \cdot u_y}$). 
 
 \smallskip
 
 The result can be stated as follows. In dimension $N$ with the canonical basis $(e_i)_{1\leq i \leq N}$ and coordinates $(x_i)_{1 \leq i \leq N}$, the spins take values in $\{ x : x_1^2 + \dots + x_N^2 =1 \}$. We consider the observables $x_{i_1} x_{i_2} \dots x_{i_k}$ for $k$ coordinates in $\{1, \dots, N \}$, the clusters $C_i$ of spins connected to the boundary by avoiding the hypersurface  $e_i^{\perp}$ and staying in the half-space $x_{i} >0$, and the boundary conditions given by the spins whose value is constant and positive in each of these $k$ directions and zero in the others. Then, we have (in the same sense as above), for a lattice point $z$,
 \begin{equation}
 \langle x_{i_1}(z) \dots x_{i_k}(z) \rangle \asymp  \P(z\stackrel{C_{i_1}}{\longleftrightarrow} \partial, \dots, z \stackrel{C_{i_k}}{\longleftrightarrow} \partial),
 \end{equation}
 where $x_i(z) := u_z \cdot e_i$. The proof relies on the same methodology as above, using now the reflections with respect to the hypersurfaces associated with $e_{i_1}, \dots, e_{i_k}$.
 
%
%
%
%

This is a direct extension (in the dimension) of the case $k=1$ and $k=2$ for the Villain model. In this latter case, $x = \cos(\theta)$, $y = \sin (\theta)$  then $xy = \sin(2\theta)/2$, the boundary condition is this time given by $\theta = \pi/4$ and we used the reflections over the axis $(Ox)$ and $(Oy)$. (Note that by the change of variable $\theta \mapsto \theta - \pi/4$, we retrieve the observable $\sin (2\theta - \pi/2)/2 = \cos(2\theta)$.)

\subsection{Questions and open problems}

\paragraph{Improvements of the results.} Proposition \ref{prop:Villain-k=1} and Proposition \ref{prop:Villain-k=2} above (and their extensions to other models) provide two-sided bounds between connectivity properties of the random clusters and spin correlations. One improvement would be to prove that the ratio of these observables converge to some positive constant. Furthermore, akin to the incipient infinite cluster (IIC) in percolation (defined as the limit of the conditional law of the system given that the origin is connected to the boundary of a box of size $N$ and $N \to \infty$ at $p=p_c$, or equivalently directly in infinite volume measure with $p \downarrow p_c$), a natural question is to prove the existence of an IIC-type limit for the Villain model. In this case the limiting ratios should be $\langle e^{ik\theta_0}\rangle_{IIC_k}$ for $k=1,2$ respectively.

\smallskip

\noindent \textbf{Higher-order spin observables.} We found two random cluster connectivity events that give two-sided uniform bounds on $\langle e^{i k \theta_x} \rangle$: for $k=1$ and $k=2$. In particular, these observables have the same exponents in the KT phase. The first natural question is to generalize this.

\smallskip

\textbf{Question:} 
Find a ``simple" random clusters representation that encodes $\langle e^{i k \theta_x} \rangle$ for $k \geq 3$, and more generally $\langle e^{i k_1 \theta_{x_1}} e^{i k_2 \theta_{x_2}}  \dots e^{i k_n \theta_{x_n}}  \rangle$.

\smallskip

Given our results on $\langle e^{i k \theta_x} \rangle$ for $k=1$ and $k=2$, it is natural to try to generalize them to $k \geq 3$, i.e. to find some connectivity events for random clusters in the extended Villain model whose probability is up to multiplicative constant comparable to $\langle  e^{i k \theta_x}\rangle$. For $k=3$, it seems  natural to consider $C_i := \{x\in\tilde{\mc G}: x\stackrel{\{\pm\alpha_i\}^c}{\longleftrightarrow}  \partial \}$ where $\alpha_1 = e^{i \frac{\pi}{6}}$, $\alpha_2 = e^{-i \frac{\pi}{6}}$ and $\alpha_3 = i$, and  the measure preserving mappings $\sigma_i$ associated with the reflections $R_i(z) = \alpha_i^2 \bar{z}$. However, complications arise when trying to generalize our methods to the clusters $C_1 \cap C_2$ or $C_1 \cap C_2 \cap C_3$. For instance, the reflections do not commute anymore as $R_1 R_2 (z) = e^{i \frac{\pi}{3}} z$ and $R_2 R_1 (z) = e^{-i \frac{\pi}{3}} z$. 

A direct generalization is to consider the following. Let $A$ be the set of $2k$-th roots of $(-1)$. To almost every configuration $\sigma$ we associate a combinatorial type $T(\sigma)$ in the following way. On each edge of $\tilde{\mc G}$, consider the points $x$ where $u_x\in A$. This consists a.s. of a finite number of perfect sets; points where $u_x=\alpha\in A$ not separated by a point where $u\in A\setminus\{\alpha\}$ are identified as a single point in the combinatorial type, labelled by $\alpha$. The order of these new vertices (each representing a perfect set) labelled by elements of $A$ is recorded, but not their position.

For each pair $\pm\alpha\in A$ we consider a transformation $\sigma_ \alpha=\sigma_{-\alpha}$ defined as follows: 
Let $C\subset\tilde{\mc G}$ denote the cluster
$$
C=\{y\in\tilde{\mc G}: y\stackrel{\{\pm\alpha\}^c}{\longleftrightarrow} x_0\},
$$
a random compact subset of $\tilde{\mc G}$; $x_0$ is a fixed point of interest. Then $\sigma_\alpha:\Omega=C_0(\tilde{\mc G},\U)\rightarrow\Omega$ is given by:
\begin{equation*}
\left\{\begin{array}{lll}
\sigma(u)_y&=u_y&{\rm if\ }y\in C\\
\sigma(u)_y&=\alpha^2\overline{u_y}&{\rm if\ }y\notin C, {\rm and\ }\partial\notin C.
\end{array}\right.
\end{equation*}
This gives $k$ measure-preserving involutions. Remark that $u_{x_0}^k$ is unchanged by application of $\sigma$ in the first case, and replaced by $-u_{x_0}^k$ in the second case. In the case $k=1$ (resp. $k=2$), these measure-preserving involutions on configuration space generate an abelian group of order 2 (resp. 4); whereas for $k\geq 3$ they generate an infinite group of measure-preserving transformations.

\bigskip

\noindent \textbf{Exponent interpolation.} The interpolation of the events considered for $k=1$ and $k=2$ can be understood as ``splitting and  continuously  sliding" the semi-circle $(-i,i)$ in two semi-circles $(-\xi,\xi)$ and $(\bar{\xi},- \bar{\xi})$ via $\theta \mapsto (-i e^{-i \theta }, i e^{-i \theta})$ and  $\theta \mapsto (-i e^{i \theta }, i e^{i \theta})$ on $[0, \frac{\pi}{4} ]$. Given this observation, the following question naturally arises.
 
\smallskip

 \textbf{Question:} Does the analogue of Proposition \ref{prop:Villain-k=2} hold for non-integer $1 < k < 2$, now with the semi-circles $(-i e^{i \theta }, i e^{i \theta})$ and $(-i e^{-i \theta }, i e^{-i \theta})$ and $\theta = \frac{\pi}{2k} $? 
 
 \smallskip
 
Our method of proof does not seem to apply directly to this extension, for the same difficulties as those mentioned in Question 1. There are other natural ways to define continuous families of critical exponents in the extended Villain model; in particular one can consider cluster connectivities $x\stackrel{S}{\longleftrightarrow} y$ (as in \eqref{eq:clusterdef}), where $S$ is a circular arc $\exp(i(-\alpha,\alpha))$ with $\alpha\in (0,\pi)$. This can also be made sense of if $\alpha\geq\pi$ by lifting spins to the universal cover; i.e. one can consider points that are connected to the boundary by a simple path along which $\theta$ has a real lift staying in $(-\alpha,\alpha)$.

More generally, in the XY model, there is only a countable family of spin operators (the $e^{ik\theta}$, $k\in\Z$); so that e.g. in \eqref{eq:KTconj}, the LHS only makes sense for $k_j\in\Z$, whereas the RHS is defined for $k_j\in\R$. One can further ask if there are cluster connectivity or geometric events corresponding to such non-integer $k_j$'s.

\bigskip

\noindent \textbf{Convergence of interfaces.} 
The critical Ising and FK interfaces are well known to converge to the Schramm-Loewner Evolutions and Conformal Loop Ensembles with parameter $\kappa = 3$ and  $\kappa = 16/3$ \cite{Ising-SLE, Ising-CLE}. Extending this to $q\in [1,4)$ remains a tantalizing problem. Are there similar results for interfaces of spins/random clusters in the Villain and XY models? To be more precise, we raise the following

 \textbf{Question:} Consider the extended Villain model with boundary condition $\theta =0$. Establish and describe the scaling limit of the loops obtained as boundaries of the cluster $\{ x ~ : ~  x\stackrel{\{\pm i\}^c}{\longleftrightarrow} \partial \}$, in the low temperature phase.
 
 \smallskip
 
A usual first step to study  the scaling limit of such interfaces is to derive Russo-Seymour-Welsh (RSW) type estimates in order to prove the existence of non-trivial subsequential limits via the Aizenman-Burchard criterion. One can be interested more generally in the percolation properties of these random cluster models, such as presence/absence of FKG inequality, existence of infinite clusters, etc. In the specific case of the XY and Heisenberg models, the FKG inequality was established in \cite{Chayes-1, Chayes-2}.  Note that, as the temperature goes to zero, the edge density of clusters goes to 1; however we still expect RSW estimates to hold. Furthermore, it is natural to expect that this scaling limit is related with the Conformal Loop Ensemble as well, which is characterized by conformal invariance and a domain Markov property.

For the critical Ising model, two independent rigorous approaches have been developed to study the conformal symmetries of the spin correlations: \cite{bosonization} proceeds by equating the product of two Ising correlators with a free field (bosonic) correlator; the proof \cite{CHI} was based on convergence results for fermionic observables, used in earlier works to study the conformal invariance and the convergence of critical Ising interfaces. In the KT phase considered here for continuous spin systems, due to the Fr\"ohlich-Spencer conjecture, the spins correlations should be conformally covariant,  and the conformal invariance of the scaling limit of these interfaces should also reasonably hold.  The Hausdorff dimension of the scaling limit of the interfaces should be encoded by the asymptotics of the two-point spin correlation function, in particular by the effective temperature of the associated limiting GFF.

However, the domain Markov property should a priori not hold: given the geometric position of the interface, the restriction of the spin field to the separated connected components are not independent, there is an extra randomness due to the value of the boundary spins (they are either vertical, or horizontal, see Figure \ref{fig:villain-reflection}), which is for instance not the case in the Ising model. Our situation is somewhat closer to the one of the XOR Ising model (where the spin field $\sigma_1 \sigma_2$ is given by the product of two independent Ising models $(\sigma_1)$ and $(\sigma_2)$) at criticality. Indeed, the interfaces between $+$ and $-$ spins in this model corresponds to interfaces between $\sigma_1 = \sigma_2$ and $\sigma_1 \neq \sigma_2$. In the pre-limit, the domain Markov property does not hold since given the position of the loop, the distribution of the restriction of the spin field to the separated connected components are not independent.

To further draw the comparison, we need to recall the notion of two-valued local sets of a GFF $\phi$, denoted by $A_{-a,b}(\phi)$, which are formally speaking points that can be reached from the boundary by staying in the set of level lines of the free field $\phi$ in $[-a,b]$. For a particular case, the  boundaries of this set has the law of CLE$_4$. Furthermore, an other particular case was conjectured by Wilson \cite{Wilson} to describe the scaling limit of the interfaces of the 2d critical XOR-Ising model: interfaces of the spin field $\sigma_1 \sigma_2$. Although not proved yet, \cite{DC-Lis-Qian, DC-Lis-Qian-v2} showed that this scaling limit holds for the critical double random current model on the square lattice. 

The general version of these local sets, $A_{-a,b}(\phi)$, was rigorously introduced and studied  in \cite{ASW, AS} (the latter work proved in particular that the heights of the loops are independent only in one case), and their Hausdorff dimension was computed in \cite{SSV} via a relation with (the real part of) an imaginary chaos measure, it was earlier shown in \cite{JSW} that the renormalized scaling limit of the spin configuration of an XOR-Ising model with +/+ boundary condition agrees in law with the real part of an imaginary chaos.

\bibliographystyle{abbrv}
\bibliographystyle{alpha}
\bibliography{refs}

\begin{thebibliography}{10}

\bibitem{AHPS}
M.~{Aizenman}, M.~{Harel}, R.~{Peled}, and J.~{Shapiro}.
\newblock {Depinning in integer-restricted Gaussian Fields and BKT phases of
  two-component spin models}.
\newblock {\em arXiv:2110.09498}, 2021.

\bibitem{ALS20}
J.~Aru, T.~Lupu, and A.~Sep\'{u}lveda.
\newblock The first passage sets of the 2{D} {G}aussian free field: convergence
  and isomorphisms.
\newblock {\em Comm. Math. Phys.}, 375(3):1885--1929, 2020.

\bibitem{AS}
J.~Aru and A.~Sep\'{u}lveda.
\newblock Two-valued local sets of the 2{D} continuum {G}aussian free field:
  connectivity, labels, and induced metrics.
\newblock {\em Electron. J. Probab.}, 23:Paper No. 61, 35, 2018.

\bibitem{ASW}
J.~Aru, A.~Sep\'{u}lveda, and W.~Werner.
\newblock On bounded-type thin local sets of the two-dimensional {G}aussian
  free field.
\newblock {\em J. Inst. Math. Jussieu}, 18(3):591--618, 2019.

\bibitem{Bauerschmidt-dGFF-1}
R.~{Bauerschmidt}, J.~{Park}, and P.-F. {Rodriguez}.
\newblock {The Discrete Gaussian model, I. Renormalisation group flow at high
  temperature}.
\newblock {\em arXiv:2202.02286}, 2022.

\bibitem{Baxter-Chacon}
J.~R. Baxter and R.~V. Chacon.
\newblock The equivalence of diffusions on networks to {B}rownian motion.
\newblock In {\em Conference in modern analysis and probability ({N}ew {H}aven,
  {C}onn., 1982)}, volume~26 of {\em Contemp. Math.}, pages 33--48. Amer. Math.
  Soc., Providence, RI, 1984.

\bibitem{Ising-CLE}
S.~Benoist and C.~Hongler.
\newblock The scaling limit of critical {I}sing interfaces is {CLE}$_3$.
\newblock {\em Ann. Probab.}, 47(4):2049--2086, 2019.

\bibitem{Berezinsky:1970fr}
V.~L. Berezinsky.
\newblock {Destruction of long range order in one-dimensional and
  two-dimensional systems having a continuous symmetry group. I. Classical
  systems}.
\newblock {\em Sov. Phys. JETP}, 32:493--500, 1971.

\bibitem{Blume}
M.~Blume.
\newblock Theory of the first-order magnetic phase change in
  {U${\mathrm{O}}_{2}$}.
\newblock {\em Phys. Rev.}, 141:517--524, Jan 1966.

\bibitem{BMhandbook}
A.~N. Borodin and P.~Salminen.
\newblock {\em Handbook of {B}rownian motion---facts and formulae}.
\newblock Probability and its Applications. Birkh\"{a}user Verlag, Basel,
  second edition, 2002.

\bibitem{Chayes-2}
M.~Campbell and L.~Chayes.
\newblock The isotropic {${\rm O}(3)$} model and the {W}olff representation.
\newblock {\em J. Phys. A}, 31(13):L255--L259, 1998.

\bibitem{Capel}
H.~Capel.
\newblock On the possibility of first-order transitions in {I}sing systems of
  triplet ions with zero-field splitting {II}.
\newblock {\em Physica}, 33(2):295--331, 1967.

\bibitem{Chandra}
K.~Chandrasekharan.
\newblock {\em Elliptic functions}, volume 281 of {\em Grundlehren der
  mathematischen Wissenschaften [Fundamental Principles of Mathematical
  Sciences]}.
\newblock Springer-Verlag, Berlin, 1985.

\bibitem{Chayes-1}
L.~Chayes.
\newblock Discontinuity of the spin-wave stiffness in the two-dimensional
  {$XY$} model.
\newblock {\em Comm. Math. Phys.}, 197(3):623--640, 1998.

\bibitem{Ising-SLE}
D.~Chelkak, H.~Duminil-Copin, C.~Hongler, A.~Kemppainen, and S.~Smirnov.
\newblock Convergence of {I}sing interfaces to {S}chramm's {SLE} curves.
\newblock {\em C. R. Math. Acad. Sci. Paris}, 352(2):157--161, 2014.

\bibitem{CHI}
D.~Chelkak, C.~Hongler, and K.~Izyurov.
\newblock Conformal invariance of spin correlations in the planar {I}sing
  model.
\newblock {\em Ann. of Math. (2)}, 181(3):1087--1138, 2015.

\bibitem{Omri-Peled}
O.~Cohen-Alloro and R.~Peled.
\newblock Rarity of extremal edges in random surfaces and other theoretical
  applications of cluster algorithms.
\newblock {\em Ann. Appl. Probab.}, 30(5):2439--2464, 2020.

\bibitem{bosonization}
J.~{Dub{\'e}dat}.
\newblock {Exact bosonization of the Ising model}.
\newblock {\em arXiv:1112.4399}, 2011.

\bibitem{DC-Lis-Qian}
H.~{Duminil-Copin}, M.~{Lis}, and W.~{Qian}.
\newblock {Conformal invariance of double random currents II: tightness and
  properties in the discrete}.
\newblock {\em arXiv:2107.12880}, 2021.

\bibitem{DC-Lis-Qian-v2}
H.~{Duminil-Copin}, M.~{Lis}, and W.~{Qian}.
\newblock {Conformal invariance of double random currents II: tightness and
  properties in the discrete}.
\newblock {\em arXiv e-prints}, page arXiv:2107.12880, July 2021.

\bibitem{Edwards-Sokal}
R.~G. Edwards and A.~D. Sokal.
\newblock Generalization of the {F}ortuin-{K}asteleyn-{S}wendsen-{W}ang
  representation and {M}onte {C}arlo algorithm.
\newblock {\em Phys. Rev. D}, 38:2009--2012, Sep 1988.

\bibitem{Enriquez-Kifer}
N.~Enriquez and Y.~Kifer.
\newblock Markov chains on graphs and {B}rownian motion.
\newblock {\em J. Theoret. Probab.}, 14(2):495--510, 2001.

\bibitem{Folz}
M.~Folz.
\newblock Volume growth and stochastic completeness of graphs.
\newblock {\em Trans. Amer. Math. Soc.}, 366(4):2089--2119, 2014.

\bibitem{FV-book}
S.~Friedli and Y.~Velenik.
\newblock {\em Statistical mechanics of lattice systems}.
\newblock Cambridge University Press, Cambridge, 2018.
\newblock A concrete mathematical introduction.

\bibitem{FS81}
J.~Fr\"{o}hlich and T.~Spencer.
\newblock The {K}osterlitz-{T}houless transition in two-dimensional abelian
  spin systems and the {C}oulomb gas.
\newblock {\em Comm. Math. Phys.}, 81(4):527--602, 1981.

\bibitem{FS83}
J.~Fr\"{o}hlich and T.~Spencer.
\newblock The {B}ere\v{z}inski\u{\i}-{K}osterlitz-{T}houless transition
  (energy-entropy arguments and renormalization in defect gases).
\newblock In {\em Scaling and self-similarity in physics ({B}ures-sur-{Y}vette,
  1981/1982)}, volume~7 of {\em Progr. Phys.}, pages 29--138. Birkh\"{a}user
  Boston, Boston, MA, 1983.

\bibitem{Garban-reconstruction}
C.~{Garban} and A.~{Sep{\'u}lveda}.
\newblock {Statistical reconstruction of the Gaussian free field and KT
  transition}.
\newblock {\em arXiv:2002.12284. To appear in J. Eur. Math. Soc.}

\bibitem{Garban-quantitative}
C.~{Garban} and A.~{Sep{\'u}lveda}.
\newblock {Quantitative bounds on vortex fluctuations in $2d$ Coulomb gas and
  maximum of the integer-valued Gaussian free field}.
\newblock {\em arXiv:2012.01400}, 2020.

\bibitem{Grimmett-Blume-Capel}
B.~T. Graham and G.~R. Grimmett.
\newblock Random-cluster representation of the {B}lume-{C}apel model.
\newblock {\em J. Stat. Phys.}, 125(2):287--320, 2006.

\bibitem{Grimmett}
G.~Grimmett.
\newblock {\em The random-cluster model}, volume 333 of {\em Grundlehren der
  mathematischen Wissenschaften [Fundamental Principles of Mathematical
  Sciences]}.
\newblock Springer-Verlag, Berlin, 2006.

\bibitem{JSW}
J.~Junnila, E.~Saksman, and C.~Webb.
\newblock Imaginary multiplicative chaos: moments, regularity and connections
  to the {I}sing model.
\newblock {\em Ann. Appl. Probab.}, 30(5):2099--2164, 2020.

\bibitem{KP17}
V.~{Kharash} and R.~{Peled}.
\newblock {The Fr{\"o}hlich-Spencer Proof of the
  Berezinskii-Kosterlitz-Thouless Transition}.
\newblock {\em arXiv:1711.04720}, 2017.

\bibitem{Kosterlitz_1973}
J.~M. Kosterlitz and D.~J. Thouless.
\newblock Ordering, metastability and phase transitions in two-dimensional
  systems.
\newblock {\em Journal of Physics C: Solid State Physics}, 6(7):1181--1203, apr
  1973.

\bibitem{Lupu_inter}
T.~Lupu.
\newblock From loop clusters and random interlacements to the free field.
\newblock {\em Ann. Probab.}, 44(3):2117--2146, 2016.

\bibitem{Newman-Wu}
C.~M. Newman and W.~Wu.
\newblock Lee-{Y}ang property and {G}aussian multiplicative chaos.
\newblock {\em Comm. Math. Phys.}, 369(1):153--170, 2019.

\bibitem{Peled-LN}
R.~Peled and Y.~Spinka.
\newblock Lectures on the {S}pin and {L}oop {O}(n) {M}odels.
\newblock In V.~Sidoravicius, editor, {\em Sojourns in Probability Theory and
  Statistical Physics - I}, pages 246--320, Singapore, 2019. Springer
  Singapore.

\bibitem{Dilute-Potts}
X.~Qian, Y.~Deng, and H.~W.~J. Bl\"ote.
\newblock Dilute {P}otts model in two dimensions.
\newblock {\em Phys. Rev. E}, 72:056132, Nov 2005.

\bibitem{SSV}
L.~Schoug, A.~Sep\'{u}lveda, and F.~Viklund.
\newblock Dimensions of two-valued sets via imaginary chaos.
\newblock {\em Int. Math. Res. Not. IMRN}, (5):3219--3261, 2022.

\bibitem{SW-algorithm}
R.~H. Swendsen and J.-S. Wang.
\newblock Nonuniversal critical dynamics in {M}onte {C}arlo simulations.
\newblock {\em Phys. Rev. Lett.}, 58:86--88, Jan 1987.

\bibitem{EM21}
D.~{van Engelenburg} and M.~{Lis}.
\newblock {An elementary proof of phase transition in the planar XY model}.
\newblock {\em arXiv:2110.09465}, 2021.

\bibitem{Wilson}
D.~B. {Wilson}.
\newblock {XOR-Ising Loops and the Gaussian Free Field}.
\newblock {\em arXiv:1102.3782}, 2011.

\bibitem{Wirth19}
M.~{Wirth}.
\newblock {Maximum of the integer-valued Gaussian free field}.
\newblock {\em arXiv:1907.08868}, 2019.

\bibitem{Wolff}
U.~Wolff.
\newblock Collective {M}onte {C}arlo {U}pdating for {S}pin {S}ystems.
\newblock {\em Phys. Rev. Lett.}, 62:361--364, Jan 1989.

\end{thebibliography}

\end{document}